\RequirePackage{fix-cm}
\documentclass[smallextended]{svjour3}       \smartqed  \usepackage{graphicx}
\usepackage{amssymb}
\usepackage{amsmath}
\usepackage{bbm}
\usepackage{subfigure}
\newtheorem{prop}{Proposition}[section]
\newtheorem{assumption}{Assumption}[section]
\journalname{Numerical Algorithms}
\begin{document}

\title{A Structural Analysis of Field/Circuit Coupled Problems Based on a Generalised Circuit Element}

\author{Idoia Cortes Garcia         \and
        Herbert De Gersem 			\and
        Sebastian Sch\"ops
}

\institute{Idoia Cortes Garcia, Herbert De Gersem and Sebastian Sch\"ops \at
              Technische Universit\"at Darmstadt,  64289 Darmstadt, Germany\\
              Tel.: +49 6151 16 - 24392\\
              Fax: +49 6151 16 - 24404\\
              \email{cortes@gsc.tu-darmstadt.de}, \email{degersem@temf.tu-darmstadt.de}, \email{schoeps@temf.tu-darmstadt.de}   
}

\date{Received: date / Accepted: date}

\maketitle

\begin{abstract}
In some applications there arises the need of a spatially distributed description of a physical quantity inside a device coupled to a circuit. Then, 
the in-space discretised system of partial differential equations is coupled to the system of equations describing the circuit (Modified Nodal 
Analysis) which yields a system of Differential Algebraic Equations (DAEs).
This paper deals with the differential index analysis of such coupled systems. For that, a new generalised inductance-like element is defined. The 
index of the DAEs obtained from a circuit containing such an element is then related to the topological characteristics of the circuit's 
underlying graph.  
Field/circuit coupling is performed when circuits are simulated containing elements described by Maxwell's equations.
The index of such 
systems with two different types of magnetoquasistatic formulations (A* and T-$\Omega$) is then deduced by showing that the spatial discretisations 
in both cases 
lead to an inductance-like 
element.
\end{abstract}

\section{Introduction}
\label{intro}
Classically in applications with electrical circuitry, they are modelled as networks of which elements are described by lumped  
models. Those elements idealise the behaviour of the spatially distributed devices to simple algebraic or differential relations between currents and 
voltages. Most 
circuit solvers use Modified Nodal Analysis (MNA) to describe the circuit topolgy \cite{Ho_1975aa},
which together with the device models leads to a system of Differential Algebraic Equations (DAEs) 
\cite{Lamour_2013aa,Gunther_1995aa,Estevez-Schwarz_2000aa}.
In many applications standard MNA lumped-element models are not precise enough. This happens when some of the devices can only be treated by detailed 
field models, e.g. to resolve nonlinear wave propagation effects or frequency-dependent current distributions. An example of such an application 
is an electrical drive, consisting of a machine, a power converter and a control device. 

Spatially distributed electromagnetic fields are described by Maxwell's equations \cite{Maxwell_1864aa,Jackson_1998aa}. Those are a set of 
Partial Differential Equations that 
can be numerically simulated by techniques such as the Finite Element Method \cite{Monk_2003aa} or the Finite Integration Technique 
\cite{Weiland_1977aa}. However, 
simulating 
all devices and the interconnecting circuit with such a method is prohibitively costly. A better approach is to couple the system of equations 
describing the 
circuit 
with the spatially discretised systems of equations describing the fields inside some selected devices 
\cite{Tsukerman_1993ab,Bartel_2011aa,Bedrosian_1993aa}. This 
is 
sometimes called refined modelling \cite{Bartel_2007aa}.

Under low frequency assumptions, Maxwell's equations can be simplified to a magnetoquasistatic setting. Typically, the remaining equations are 
combined into a formulation
by defining 
appropriate 
potentials. Depending on the choice of potentials, different formulations are obtained \cite{Biro_1995aa,Webb_1993aa}, such as the A* and the 
T-$\Omega$ formulations. After spatial discretisation 
of the 
magnetoquasistatic approximation, a system of DAEs
is obtained \cite{Cortes-Garcia_2018aa}. 

The coupling of the field and circuit systems leads to a coupled system of DAEs whose numerical (and analytical) complexity can be described by 
the 
notion of its index
\cite{Petzold_1982aa}. There are several index definitions \cite{Mehrmann_2015aa}. This paper focuses on the field/circuit coupled system's 
differential index and develops theoretical results that allow to deduce the index by studying the properties of the field's subsystem 
of equations and the topological features of the circuit, only. For that 
purpose a new generalised inductance-like element is defined and index results of a circuit containing these elements are deduced.
Both common magnetoquasistatic formulations are coupled to a circuit described by MNA and are shown to be important examples of such a type of 
element. The new simplified analysis for the A* formulation confirms known results from literature (see e.g. \cite{Bartel_2011aa,Nicolet_1996aa}), 
while the analysis of the T-Omega case is new. Related works, e.g. \cite{Tsukerman_2002aa}, are neither considering MNA nor 3D models.

The paper is structured as follows: Section~\ref{sec:DAE} is a brief introduction into DAEs and defines the differential index.   
Section~\ref{sec:Circuit}
introduces the MNA, important concepts for its index study and a new generalised definition of a circuit element with novel theoretical results. The 
field 
equations, different formulations thereof, field-circuit couplings and spatial discretisation are presented in Section~\ref{sec:field}. 
Section~\ref{sec:index} 
states the index results for the coupled system.
Finally, 
in Section~\ref{sec:numeric} some numerical results are shown and Section~\ref{sec:conclusions} draws the conclusions.

\section{Differential Algebraic Equations}\label{sec:DAE}
A system of differential algebraic equations has the form
\begin{equation}\label{eq:dae}
\mathbf{F}(\ensuremath{ \mathbf{x} }', \ensuremath{ \mathbf{x} }, t)= 0,
\end{equation}
with $\ensuremath{ \mathbf{x} }' = \frac{\mathrm{d}\ensuremath{ \mathbf{x} }}{\mathrm{d}t}$, $\det\left(\frac{\partial \mathbf{F}}{\partial 
\ensuremath{ \mathbf{x} }'}\right) = 0$ and $\mathbf{x}:\mathcal{I}=[t_0,\ 
t_\mathrm{end}]\rightarrow\mathbbm{R}^n$.

The index (see \cite{Brenan_1995aa}) allows to classify a system of DAEs according to its numerical and analytical complexity. Even though there 
are 
several 
different index types (e.g. the perturbation, nilpotency or tractability index), they all coincide in the case of linear DAEs \cite{Brenan_1995aa}. 
For the analysis in the paper, the 
differential 
index concept is used. It can 
be intuitively thought of as a way of measuring how far away the system is of an ordinary differential equation (ODE) in terms of differentiation. 
The higher the index of a DAE is, the more difficulties arise when treating te system numerically or analytically. Higher index DAEs are for example 
more difficult to initialise or have a higher sensitivity towards small perturbations.
\begin{definition}[Differential index \cite{Brenan_1995aa}] 
	The system of DAEs \eqref{eq:dae} has differential index $m$, if $m$ is the minimal number of differentiations
	\begin{align}\label{eq:daeindex}
	\frac{\mathrm{d}}{\mathrm{d}t}\mathbf{F}(\ensuremath{ \mathbf{x} }', \ensuremath{ \mathbf{x} }, t),  \cdots,  
	\frac{\mathrm{d}^{(m)}}{\mathrm{d}t^{(m)}}\mathbf{F}(\ensuremath{ \mathbf{x} }', \ensuremath{ \mathbf{x} }, t)
	\end{align}
	needed, such that one can write a system of ordinary differential equations
	\begin{align*}
	\ensuremath{ \mathbf{x} }' = \Phi(\ensuremath{ \mathbf{x} }, t),
	\end{align*}
	with $\Phi$ being a continuous function in $\ensuremath{ \mathbf{x} }$ and $t$, only with algebraic manipulations of equations 
	\eqref{eq:dae}-\eqref{eq:daeindex}.
\end{definition}

\begin{assumption}[Smoothness]\label{ass:diff}
	In our differential index analysis we assume that all the functions involved in the studied system are sufficiently differentiable.
\end{assumption}
For a relaxation of Assumption \ref{ass:diff}, other index definitions should be used, such as the tractability index (see e.g. 
\cite{Estevez-Schwarz_2000aa}).

\section{Circuit System}\label{sec:Circuit}
Modern electrical circuit simulators use Modified Nodal Analysis \cite{Reis_2014aa}, where the circuit is modelled as a directed graph with 
an incidence 
matrix $\mathbf{A}$. We consider circuits containing capacitors (C), inductors (L), resistors (R) and voltage (V) and current (I) sources.
Using Kirchhoff's Current Law and the lumped parameter models describing the devices at the circuit branches by algebraic 
or differential equations \cite{Reis_2014aa}, the system of DAEs of the conventional MNA \cite{Estevez-Schwarz_2000aa} is obtained, 
\begin{align}
\ensuremath{ \mathbf{A} }_\mathrm{C}\frac{\mathrm{d}}{\mathrm{d}t}\mathbf{q}(\ensuremath{ \mathbf{A} }_\mathrm{C}^{\top} \mathbf{e}, t) + 
\ensuremath{ \mathbf{A} }_\mathrm{R}\mathbf{g}(\ensuremath{ \mathbf{A} }_\mathrm{R}^{\top}\mathbf{e}, t)+ 
\ensuremath{ \mathbf{A} }_\mathrm{L}\ensuremath{ \mathbf{i} }_\mathrm{L} + 
\ensuremath{ \mathbf{A} }_\mathrm{V}\ensuremath{ \mathbf{i} }_\mathrm{V} + 
\ensuremath{ \mathbf{A} }_\mathrm{I}\ensuremath{ \mathbf{i} }_\mathrm{src}(t) &= 0 \nonumber \\
\frac{\mathrm{d}}{\mathrm{d}t}\phi_\mathrm{L}(\ensuremath{ \mathbf{i} }_{\mathrm{L}},t)   - \ensuremath{ \mathbf{A} }_\mathrm{L}^{\top} \mathbf{e} &= 
0 \label{eq:classicmna} \\
\ensuremath{ \mathbf{A} }_\mathrm{V}^{\top} \mathbf{e} - \ensuremath{ \mathbf{v} }_{\mathrm{src}}(t) &= 0, \nonumber
\end{align} 
for $t\in\mathcal{I}=[t_0, t_\mathrm{end}]\subset \mathbbm{R}$.
Here, $\ensuremath{ \mathbf{A} }_\star$ represents the columns of the incidence matrix attributed to branches that contain a specific device and 
$\ensuremath{ \mathbf{A} } = 
[\ensuremath{ \mathbf{A} }_\mathrm{C} \ \ensuremath{ \mathbf{A} }_\mathrm{R} \ \ensuremath{ \mathbf{A} }_\mathrm{L} \ \ensuremath{ \mathbf{A} 
}_\mathrm{V} \ \ensuremath{ \mathbf{A} }_\mathrm{I}]$, 
$\mathbf{e}:\mathcal{I}\rightarrow\mathbbm{R}^{n_\mathrm{e}}$ is 
the vector of node potentials, 
$\ensuremath{ \mathbf{i} }_\star:\mathcal{I}\rightarrow\mathbbm{R}^{n_\star}$ the vector of currents through branches containing element $\star$ and 
$\mathbf{q}(\cdot)$, 
$\mathbf{g}(\cdot)$, 
$\phi_\mathrm{L}(\cdot)$, $\ensuremath{ \mathbf{i} }_\mathrm{src}(\cdot)$, $\ensuremath{ \mathbf{v} }_{\mathrm{src}}(\cdot)$ 
are functions of the lumped parameter models (C, R, L, I, V). 
The voltage across a branch can be extracted with $\ensuremath{ \mathbf{v} }_\star = \ensuremath{ \mathbf{A} }_\star^{\top} \mathbf{e}$.

When discussing the characteristics of circuits, two concepts that describe the topological properties of the underlying graph are used
\cite{Gunther_1995aa,Reis_2014aa}:
\begin{definition}[Cutset, loop]
	\cite[Appendix A.1]{Tischendorf_2003aa} Given a graph $G = (V,E)$, with $V$ being the set of all nodes and $E$ the set of all edges,  we define
	\begin{itemize}
		\item a \emph{cutset} as a set of branches $E_\mathrm{c}$ such that its deletion from graph $G$, $G'(V, E\backslash E_\mathrm{c})$ 
		results in a disconnected 
		graph and adding any branch $e_\mathrm{c}\in E_\mathrm{c}$  to $G'$, again leads to a connected graph.
		\item a \emph{loop} as a subgraph $G_\mathrm{l}$ such that it is connected and every node $v_\mathrm{l}$ in $G_\mathrm{l}$ connects exactly 
		two edges of $G_\mathrm{l}$ with each other.
	\end{itemize}
\end{definition}
As the coupling of a field model with a circuit is studied, we introduce a new notation for the branches representing generalised elements 
that will be defined later and that describes our field models. From 
now on, columns of the incidence matrix, currents and voltages corresponding to generalised elements will be denoted by the 
subscript $\lambda$.

In order to ensure existence and uniqueness of the circuit's solution, the following properties for its topology and functions are assumed 
(see 
\cite{Estevez-Schwarz_2000aa}).
\begin{assumption}[Well posedness]\label{ass:unique}
	The MNA circuit equations fulfil
	\begin{itemize}
		\item there are no cutsets containing only current sources, that is,
		\begin{equation*}
		\ker(\ensuremath{ \mathbf{A} }_\mathrm{R} \ \ensuremath{ \mathbf{A} }_\mathrm{C} \ \ensuremath{ \mathbf{A} }_\mathrm{V} \ \ensuremath{ 
		\mathbf{A} }_\mathrm{L} \ \ensuremath{ \mathbf{A} }_\lambda)^{\top} = \{0\}.
		\end{equation*}
		\item there are no loops of voltage sources
		\begin{equation*}
		\ker \ensuremath{ \mathbf{A} }_\mathrm{V} = \{0\}.
		\end{equation*}
		\item the functions describing conductances, inductances and capacitances 
		\begin{align*}
		\mathbf{G}(\ensuremath{ \mathbf{v} }_\mathrm{R}, t)= 
		\frac{\partial \mathbf{g}(\ensuremath{ \mathbf{v} }_\mathrm{R}, t)}{\partial \ensuremath{ \mathbf{v} }_\mathrm{R}}, &&
		\mathbf{L}(\ensuremath{ \mathbf{i} }_\mathrm{L}, t) = \frac{\partial \phi(\ensuremath{ \mathbf{i} }_\mathrm{L}, t)}{\partial \ensuremath{ 
		\mathbf{i} }_{\mathrm{L}}} \text{ and} &&
		\mathbf{C}(\ensuremath{ \mathbf{v} }_\mathrm{C}, t) = \frac{\partial \mathbf{q}(\ensuremath{ \mathbf{v} }_\mathrm{C}, t)}{\partial 
			\ensuremath{ \mathbf{v} }_\mathrm{C}}
		\end{align*}
		are positive definite.
	\end{itemize}
\end{assumption}
The index study of system \eqref{eq:classicmna} under Assumption \ref{ass:unique} has already been carried out e.g. in \cite{Estevez-Schwarz_2000aa}. 
However, a new generalised element is now introduced, which admits the coupling of more complex element-types.

\subsection{Generalised Element}

The following section presents the definition of the generalised element and concludes with index results of the system of DAEs that results when
describing a circuit that contains such elements. This allows to give index statements about circuit systems coupled to DAEs arising from 
refined 
models without having the need of analysing the overall coupled system.

\begin{definition}[Inductance-like element]\label{def:inductlike}
	We define an {inductance-like} element as one described by a DAE
	\begin{equation}
	\ensuremath{ \mathbf{F} }\left(\frac{\mathrm{d}}{\mathrm{d}t}\ensuremath{ \mathbf{x} }_\lambda, \frac{\mathrm{d}}{\mathrm{d}t}\ensuremath{ 
	\mathbf{i} }_\lambda , \ensuremath{ \mathbf{x} }_\lambda, \ensuremath{ \mathbf{i} }_\lambda, \ensuremath{ \mathbf{v} }_\lambda,t\right) = 
	0,\label{eq:induct1}
	\end{equation}
	with $\ensuremath{ \mathbf{x} }_\lambda:\mathcal{I}\rightarrow\mathbbm{R}^{n_\mathrm{dof}}$ and $\ensuremath{ \mathbf{i} }_\lambda, \ensuremath{ 
	\mathbf{v} }_\lambda:\mathcal{I}\rightarrow\mathbbm{R}^{n_\lambda}$ , 
	such that at most one differentiation 
	\begin{equation}
	\frac{\mathrm{d}}{\mathrm{d}t}\ensuremath{ \mathbf{F} }\left(\frac{\mathrm{d}}{\mathrm{d}t}\ensuremath{ \mathbf{x} }_\lambda, 
	\frac{\mathrm{d}}{\mathrm{d}t}\ensuremath{ \mathbf{i} }_\lambda , 
	\ensuremath{ \mathbf{x} }_\lambda, \ensuremath{ \mathbf{i} }_\lambda, \ensuremath{ \mathbf{v} }_\lambda, t\right) = 0 \label{eq:induct2}
	\end{equation}
	is needed to obtain from equations \eqref{eq:induct1}-\eqref{eq:induct2}   a system of the form
	\begin{align}
	\frac{\mathrm{d}}{\mathrm{d}t}\ensuremath{ \mathbf{x} }_\lambda &= \mathbf{f}_{\ensuremath{ \mathbf{x} }}(\ensuremath{ \mathbf{x} }_\lambda, 
	\ensuremath{ \mathbf{i} }_\lambda, \ensuremath{ \mathbf{v} }_\lambda, t)\\
	\frac{\mathrm{d}}{\mathrm{d}t}\phi(\ensuremath{ \mathbf{i} }_\lambda, \ensuremath{ \mathbf{x} }_\lambda, t)&= \mathbf{{f}}_{\phi}(\ensuremath{ 
	\mathbf{x} }_\lambda, \ensuremath{ \mathbf{i} }_\lambda, \mathbf{v}_\lambda, t), 
	\label{eq:induc}
	\end{align}
	with the properties
	\begin{itemize}
		\item $\frac{\partial}{\partial \ensuremath{ \mathbf{i} }_\lambda}\phi(\ensuremath{ \mathbf{i} }_\lambda, \ensuremath{ \mathbf{x} }_\lambda, 
		t)$ is regular.
		\item 	$\frac{\partial}{\partial \mathbf{v}_\lambda}\Big(\left(\frac{\partial \phi}{\partial \ensuremath{ \mathbf{i} 
		}_\lambda}\right)^{-1}\left( 
		-\frac{\partial \phi}{\partial 
			\ensuremath{ \mathbf{x} }_\lambda}\mathbf{f}_{\ensuremath{ \mathbf{x} }} -  \frac{\partial \phi}{\partial t} + \mathbf{{f}}_{\phi}\right)
		\Big)$ is 
		positive 
		definite.
	\end{itemize}
\end{definition}

\begin{remark}
	Crucial in \eqref{eq:induct1} is that the time derivative of the branch voltage $\ensuremath{ \mathbf{v} }_\lambda$ does not appear in the 
	expression.
\end{remark}

\begin{example}
	Two examples for inductance-like devices are 
	\begin{enumerate}
		\item[(a)] classical inductances written as
		\begin{align*}
		\mathbf{v}_\lambda(t) = \ensuremath{ \mathbf{L} } \frac{\mathrm{d}}{\mathrm{d}t}\ensuremath{ \mathbf{i} }_\lambda(t),
		\end{align*}
		with $\ensuremath{ \mathbf{L} }$ positive definite. Here $\ensuremath{ \mathbf{x} }_\lambda = \{ \ \}$, $\mathbf{f}_{\ensuremath{ \mathbf{x} 
		}} = \{ \ \}$, $\phi(\ensuremath{ \mathbf{i} }_\lambda, t) = 
		\ensuremath{ \mathbf{L} }\ensuremath{ \mathbf{i} }_\lambda(t)$ and $\mathbf{f}_\phi = \ensuremath{ \mathbf{v} }_\lambda(t)$.
		\item[(b)]flux-formulated inductances with 
		\begin{align*}
		\Phi(t) &= \phi_L(\ensuremath{ \mathbf{i} }_\lambda, t)\\
		\ensuremath{ \mathbf{v} }_\lambda(t) &= \frac{\mathrm{d}}{\mathrm{d}t} \Phi(t),
		\end{align*}
		where $\ensuremath{ \mathbf{L} }(\ensuremath{ \mathbf{i} }_\lambda, t) := \frac{\partial}{\partial \ensuremath{ \mathbf{i} 
		}_\lambda}\phi_L(\ensuremath{ \mathbf{i} }_\lambda,t)$ is positive definite. Here 
		\begin{align*}
		\ensuremath{ \mathbf{x} }_\lambda&= \Phi(t) &  \mathbf{f}_{\ensuremath{ \mathbf{x} }}(\ensuremath{ \mathbf{v} }_\lambda)&= \ensuremath{ 
		\mathbf{v} }_\lambda(t)\\
		\phi(\ensuremath{ \mathbf{i} }_\lambda, t) &= \phi_\mathrm{L}(\ensuremath{ \mathbf{i} }_\lambda, t) & \mathbf{f}_{\phi}(\ensuremath{ 
		\mathbf{x} }_\lambda) &= 
		\ensuremath{ \mathbf{v} }_\lambda(t).
		\end{align*}
	\end{enumerate}
\end{example}

Like in \cite{Estevez-Schwarz_2000aa}, we define for the index study the projector $\ensuremath{ \mathbf{Q} }_\star$ onto the kernel of $\ensuremath{ 
\mathbf{A} }_\star^{\top}$ and its 
complementary $\ensuremath{ \mathbf{P} }_\star = \ensuremath{ \mathbf{I} } - 
\ensuremath{ \mathbf{Q} }_\star$, which projects onto the support of $\ensuremath{ \mathbf{A} }_\star^{\top}$.
\begin{theorem}[Circuit index]\label{theo:index}
	Given an inductance-like element $\lambda$ following Definition \ref{def:inductlike} which is coupled to a circuit 
	fulfilling 
	assumptions \ref{ass:unique}  with $\ensuremath{ \mathbf{v} }_\lambda = \ensuremath{ \mathbf{A} }_\lambda^{\top} \mathbf{e}$, then  
	the entire system has differential index
	\begin{enumerate}
		\item[(i)] 1 if there are no cutsets containing only inductors, current sources and inductance-like elements (LI$\lambda$-cutsets), that is,  
		$\ker(\ensuremath{ \mathbf{A} }_\mathrm{R} \ \ensuremath{ \mathbf{A} }_\mathrm{C} \ \ensuremath{ \mathbf{A} }_\mathrm{V})^{\top} = \{0\}$ nor 
		loops of voltage sources and capacitances only
		(CV-loops), that is, $\ker\ensuremath{ \mathbf{Q} }_\mathrm{C}^{\top}\ensuremath{ \mathbf{A} }_\mathrm{V} = \{0\} $.
		\item[(ii)] 2, otherwise.
	\end{enumerate}
\end{theorem}
\begin{proof}
	The proof is analogous to the differential index proof in \cite{Estevez-Schwarz_2000aa}, by taking into account the new terms in the system
	\begin{align*}
	\ensuremath{ \mathbf{A} }_\mathrm{C}\frac{\mathrm{d}}{\mathrm{d}t}\mathbf{q}(\ensuremath{ \mathbf{A} }_\mathrm{C}^{\top} 
	\mathbf{e}, t) + 
	\ensuremath{ \mathbf{A} }_\mathrm{R}\mathbf{g}(\ensuremath{ \mathbf{A} }_\mathrm{R}^{\top}\mathbf{e}, t)+ \ensuremath{ \mathbf{A} 
	}_\mathrm{L}\ensuremath{ \mathbf{i} }_\mathrm{L} + 
	\ensuremath{ \mathbf{A} }_\mathrm{V}\ensuremath{ \mathbf{i} }_\mathrm{V} + 
	\boxed{\ensuremath{ \mathbf{A} }_\lambda\ensuremath{ \mathbf{i} }_\lambda} + \ensuremath{ \mathbf{A} }_\mathrm{I}\ensuremath{ \mathbf{i} 
	}_\mathrm{src}(t) = 0& \\[-0.5em]
	\frac{\mathrm{d}}{\mathrm{d}t}\phi_\mathrm{L}(\ensuremath{ \mathbf{i} }_{\mathrm{L}},t) - \ensuremath{ \mathbf{A} }_\mathrm{L}^{\top} \mathbf{e} 
	= 0& \\
	\ensuremath{ \mathbf{A} }_\mathrm{V}^{\top} \mathbf{e} - \ensuremath{ \mathbf{v} }_{\mathrm{src}}(t) = 0& \\
	\boxed{\mathbf{F}\left(\frac{\mathrm{d}}{\mathrm{d}t}\ensuremath{ \mathbf{x} }_\lambda, \frac{\mathrm{d}}{\mathrm{d}t}\ensuremath{ \mathbf{i} 
	}_\lambda , \ensuremath{ \mathbf{x} }_\lambda, \ensuremath{ \mathbf{i} }_\lambda, \ensuremath{ \mathbf{A} }_\lambda^{\top} 
		\mathbf{e}\right)= 0}&
	\end{align*}
	accounting for the inductance-like elements.
\end{proof}

\begin{remark}
	The index results are valid for circuits containing multiple inductance-like elements 
	$$\mathbf{F}_1\left(\frac{\mathrm{d}}{\mathrm{d}t}\ensuremath{ \mathbf{x} }_{\lambda,1}, 
	\frac{\mathrm{d}}{\mathrm{d}t}\ensuremath{ \mathbf{i} }_{\lambda,1} , \ensuremath{ \mathbf{x} }_{\lambda,1}, \ensuremath{ \mathbf{i} 
	}_{\lambda,1}, \ensuremath{ \mathbf{A} }_{\lambda,1}^{\top} 
	\mathbf{e}\right), \ldots, \mathbf{F}_n\left(\frac{\mathrm{d}}{\mathrm{d}t}\ensuremath{ \mathbf{x} }_{\lambda,n}, 
	\frac{\mathrm{d}}{\mathrm{d}t}\ensuremath{ \mathbf{i} }_{\lambda,n} , \ensuremath{ \mathbf{x} }_{\lambda,n}, \ensuremath{ \mathbf{i} 
	}_{\lambda,n}, 
	\ensuremath{ \mathbf{A} }_{\lambda,n}^{\top} 
	\mathbf{e}\right).$$
\end{remark}

\begin{prop}[Linear index-2 components]\label{prop:giwli2}
	Let the DAE of the inductance-like element of Definition \ref{def:inductlike} have the structure
	\begin{align*}
	0 = \mathbf{F}\left(\frac{\mathrm{d}}{\mathrm{d}t}\ensuremath{ \mathbf{x} }_\lambda, \frac{\mathrm{d}}{\mathrm{d}t}\ensuremath{ \mathbf{i} 
	}_\lambda, \ensuremath{ \mathbf{x} }_\lambda, \ensuremath{ \mathbf{i} }_\lambda, \ensuremath{ \mathbf{v} }_\lambda, t\right)
	= \mathbf{\tilde{F}}\left(\frac{\mathrm{d}}{\mathrm{d}t}\ensuremath{ \mathbf{x} }_\lambda, \frac{\mathrm{d}}{\mathrm{d}t}\ensuremath{ \mathbf{i} 
	}_\lambda, \ensuremath{ \mathbf{x} }_\lambda, \ensuremath{ \mathbf{i} }_\lambda, t\right) + \mathbf{B}\ensuremath{ \mathbf{v} }_\lambda,
	\end{align*}
	with $\ensuremath{ \mathbf{B} }\in\mathbbm{R}^{(n_\mathrm{dof} + n_\lambda)\times n_\lambda}$, that is, the voltage term in the original DAE 
	system is linear,  whenever 
	the 
	inductance-like element is contained in an $LI\lambda$ cutset, then, the 
	coupled system of Theorem \ref{theo:index} has linear index-2 components. 
\end{prop}
\begin{proof}
	Analogous to the differential index proof in \cite{Estevez-Schwarz_2000aa}, one can see that the index-2 components are the node 
	potentials in $LI\lambda$ cutsets, that is, $\ensuremath{ \mathbf{Q} }_{CRV}\mathbf{e}$, where $\ensuremath{ \mathbf{Q} }_\mathrm{CRV}$ is a 
	projector onto 
	$\ker \left(\ensuremath{ \mathbf{A} }_\mathrm{C} \ \ensuremath{ \mathbf{A} }_\mathrm{R} \ \ensuremath{ \mathbf{A} }_\mathrm{V}\right)^{\top}$, 
	and the currents in branches containing voltage sources in 
	CV-loops, that is $\bar{\ensuremath{ \mathbf{Q} }}_\mathrm{V-C}\ensuremath{ \mathbf{i} }_{\mathrm{V}}$, with $\bar{\ensuremath{ \mathbf{Q} 
	}}_\mathrm{V-C}$ a projector onto 
	$\ker\ensuremath{ \mathbf{Q} }_\mathrm{C}^{\top}\ensuremath{ \mathbf{A} }_\mathrm{V}$. If the voltage $\ensuremath{ \mathbf{A} 
	}_\lambda^{\top}\mathbf{e}$  in the original DAE of the 
	inductance-like device is thus linear, then the 
	possible 
	index-2 component $\ensuremath{ \mathbf{A} }_\lambda ^{\top}\ensuremath{ \mathbf{Q} }_\mathrm{CRV}\mathbf{e}$ is linear. The other possible 
	index-2 components are part 
	of the original MNA equations and it has already been shown previously (see e.g. \cite{Baumanns_2010aa}) that they are linear.
\end{proof}
Now that we know inductance-like elements behave (from the index point-of-view) like an inductance in a circuit, two more complex examples of such 
type of 
elements with practical 
relevance will be presented in the following: the spatially discretised magnetoquasistatic models in A*  and T-$\Omega$ formulations.

\section{Refined Systems}\label{sec:field}

The electromagnetic field in a magnetoquasistatic approximation is defined by Maxwell's equations for the eddy current problem \cite{Jackson_1998aa}
\begin{subequations}\label{eq:maxwell}
	\begin{align}
	\nabla \times\ensuremath{\vec{E}} &= - \frac{\partial}{\partial t}\ensuremath{\vec{B}} \label{eq:faraday} \\
	\nabla \times \ensuremath{\vec{H}} &= \ensuremath{\vec{J}} \label{eq:ampere}\\ 
	\nabla \cdot\ensuremath{\vec{B}} &= 0, \label{eq:monopole}
	\end{align}
\end{subequations}
where the time derivative of the electric flux density is disregarded with respect to $\ensuremath{\vec{J}}$ ($\frac{\partial }{\partial 
t}\ensuremath{\vec{D}} = 0$) in 
Maxwell-Amp\`ere's equation \eqref{eq:ampere}. 
Here, 
$\ensuremath{\vec{E}}$ is the electric field strength, $\ensuremath{\vec{B}}$ the magnetic flux density, $\ensuremath{\vec{H}}$ the magnetic field 
strength and $\ensuremath{\vec{J}}$ the electric 
current density. All quantities are vector fields $\mathcal{I}\times\Omega\rightarrow \mathbbm{R}^3$ depending on time and space, where 
$\Omega\subset\mathbbm{R}^3$. The quantities are related 
via the material equations
\begin{align*}
\ensuremath{\vec{J}} = \sigma \ensuremath{\vec{E}} + \ensuremath{\vec{J}}_\mathrm{s}	&&	\ensuremath{\vec{H}} = \nu\ensuremath{\vec{B}}.
\end{align*}
The nonnegative conductivity $\sigma$ and the positive reluctivity $\nu = \mu^{-1}$ depend on space and their dependence on the fields is for now 
disregarded 
for simplicity 
of notation. $\ensuremath{\vec{J}}_\mathrm{s}$ is the source current density.

\begin{assumption}[Domain see Figure \ref{fig:domain}]\label{ass:domain}
	The domain $\Omega\subset\mathbbm{R}^3$ has two types of subdomains, the source domains $\Omega_\mathrm{s}^{(r)}$, $r =1,\dots,n_\mathrm{s}$ and 
	the 
	conducting domain $\Omega_\mathrm{c}$. They fulfil the following 
	properties.
	\begin{itemize}
		\item $\Omega$ is contractible (see \cite{Bossavit_1998aa}).
		\item All subdomains are disjoint, that is,
		\begin{align*}
		\Omega_\mathrm{s}^{(i)}\cap\Omega_\mathrm{s}^{(j)} = \emptyset, \text{ for }i\neq j && \text{and} && 
		\Omega_\mathrm{c}\cap\Omega_\mathrm{s}^{(j)} = \emptyset \text{ }\forall j.
		\end{align*}
		\item The conductivity $\sigma$ is positive in $\Omega_\mathrm{c}$ and zero everywhere else.
		\item The source current density $\ensuremath{\vec{J}}_\mathrm{s}$ is only nonzero in $\displaystyle{\Omega_\mathrm{s} = 
			\bigcup_i\Omega_\mathrm{s}^{(i)}}$.
	\end{itemize}
\end{assumption}

\begin{figure}[h]
	\centering
	\includegraphics[scale=1]{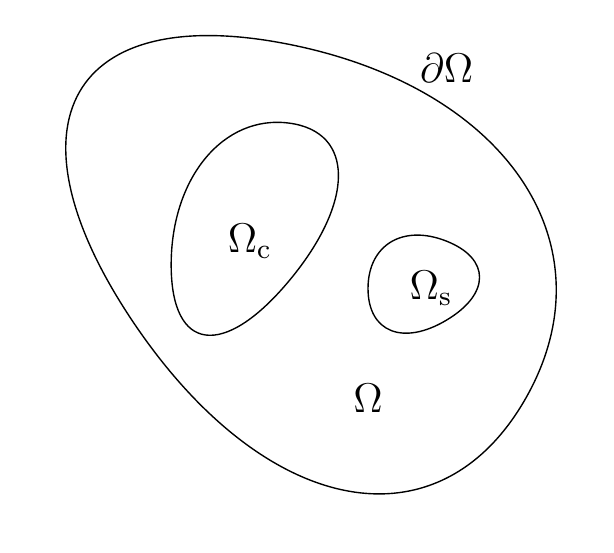}
	\caption{Sketch of domain.}\label{fig:domain}
\end{figure}

\subsection{A* and T-$\Omega$ Formulations}
Typically, Maxwell's equations are formulated by defining potentials \cite{Biro_1995aa}. In the T-$\Omega$ formulation 
\cite{Albanese_1991aa,Webb_1993aa}, 
an electric vector potential 
$\vec{T}:\mathcal{I}\times\Omega_\mathrm{c}\rightarrow\mathbbm{R}^3$ only on the conducting domain $\Omega_\mathrm{c}$
and a magnetic scalar potential  $\psi:\mathcal{I}\times\Omega\rightarrow\mathbbm{R}^3$ on the entire domain $\Omega$ are defined, such that 
\begin{align*}
\ensuremath{\vec{J}}_\mathrm{c}= \nabla \times\vec{T} && \text{and} && \ensuremath{\vec{H}} = \ensuremath{\vec{H}}_\mathrm{s} + \vec{T} -\nabla\psi,
\end{align*}
where $\ensuremath{\vec{J}}_\mathrm{c} = \sigma\ensuremath{\vec{E}}$ is the conduction current density and $\ensuremath{\vec{H}}_\mathrm{s}$ can be 
thought of as a source magnetic field 
with $\nabla \times\ensuremath{\vec{H}}_\mathrm{s} = \ensuremath{\vec{J}}_\mathrm{s}$. The following system of partial differential equations (PDEs) 
arises
\begin{equation}\label{eq:tcont}
\begin{aligned}
\nabla \times \rho \nabla \times \vec{T} + \mu \frac{\partial }{\partial t}\vec{T} - \mu \frac{\partial}{\partial t }\nabla\psi + 
\mu\frac{\partial}{\partial 
	t}\ensuremath{\vec{H}}_\mathrm{s}&=0 && \text{in }\Omega_\mathrm{c}\\
\nabla \cdot \mu \vec{T} -\nabla \cdot\left(\mu\nabla\psi\right) + \nabla \cdot\mu\ensuremath{\vec{H}}_\mathrm{s} &=0 &&\text{in }\Omega,
\end{aligned}
\end{equation}
where $\rho:\mathcal{I}\times\Omega_\mathrm{c}\rightarrow\mathbbm{R}^3$ is the specific resistance $\sigma^{-1}$.

Another possibility is the  A-$\varphi$ formulation. Here, a magnetic vector potential 
$\ensuremath{\vec{A}}:\mathcal{I}\times\Omega\rightarrow\mathbbm{R}^3$ and an 
electric scalar potential $\varphi:\mathcal{I}\times\Omega\rightarrow\mathbbm{R}$ are introduced (see \cite{Kameari_1990aa,Russenschuck_2010aa}), 
such that
\begin{align*}
\ensuremath{\vec{B}} = \nabla \times\ensuremath{\vec{A}} && \text{and} && \ensuremath{\vec{E}} = -\frac{\partial}{\partial t}\ensuremath{\vec{A}} - 
\nabla\varphi.
\end{align*}
The gauge freedom allows to choose a magnetic vector potential which leads to the A* formulation, where $\ensuremath{\vec{E}} = 
-\frac{\partial}{\partial 
	t}\ensuremath{\vec{A}}$. This yields the following PDE
\begin{equation}\label{eq:acont}
\sigma \frac{\partial }{\partial t}\ensuremath{\vec{A}} + \nabla \times\left(\nu \nabla \times \ensuremath{\vec{A}} \right) = 
\ensuremath{\vec{J}}_\mathrm{s}.
\end{equation}

In the simplest case electric boundary conditions are set at $\partial\Omega$, that is, the tangential component of the electric field is zero 
$\ensuremath{\vec{E}}_t = 0$. For 
$\vec{n}$  the unit vector normal to $\partial \Omega$ this 
translates 
into setting zero Neumann boundary conditions for the magnetic 
scalar 
potential $\psi$ 
\begin{align*}
\mu\nabla\psi\cdot\vec{n} = 0,
\end{align*}
as long as $\Omega_\mathrm{c}\cap\partial\Omega = \emptyset$ and the magnetic source function is chosen to be 
$\mu\ensuremath{\vec{H}}_\mathrm{s}\cdot\vec{n}=0$,
and zero Dirichlet boundary conditions for the magnetic vector potential $\ensuremath{\vec{A}}$
\begin{align*}
\vec{n}\times\ensuremath{\vec{A}} = 0.
\end{align*} Also the tangential component of the electric 
vector potential  $\vec{T}$ is set to zero at $\partial\Omega_\mathrm{c}$.

\subsection{Circuit Coupling}
In order to couple the three dimensional system of field equations with the zero dimensional circuit's 
equations, 
winding functions  \cite{Schops_2013aa} 
are introduced. They distribute the currents or voltages of the circuit on the field's domain $\Omega$. 

There are different types of conductor models 
that lead to winding functions with different properties \cite{Bedrosian_1993aa,Schops_2013aa}. We will consider the stranded 
conductor model, where a 
divergence-free winding function $\vec{\chi}_\mathrm{s}:\Omega_\mathrm{s}\rightarrow \mathbbm{R}^{3\times n_\mathrm{s} }$ is constructed, such that 
for 
each 
coil $j$
\begin{align*}
\ensuremath{\vec{J}}_\mathrm{s}^{(j)} = \vec{\chi}_\mathrm{s}^{(j)}i_\mathrm{s}^{(j)},
\end{align*}
where $i_\mathrm{s}^{(j)}$ is the current through the coil, $\sup(\ensuremath{\vec{J}}_\mathrm{s}^{(j)}) = \Omega_\mathrm{s}^{(j)}$ and 
\begin{equation*}
\ensuremath{\vec{J}}_\mathrm{s} = 
\sum_j\vec{\chi}_\mathrm{s}^{(j)}i_\mathrm{s}^{(j)} = \vec{\chi}_\mathrm{s}\ensuremath{ \mathbf{i} }_{\mathrm{s}}.
\end{equation*}

In the case of the T-$\Omega$ formulation, a function $\vec{\zeta}_\mathrm{s}: \Omega \rightarrow \mathbbm{R}^{3\times n_\mathrm{s} } $ is defined 
with 
$\nabla \times\vec{\zeta}_\mathrm{s} = 
\vec{\chi}_\mathrm{s}$ and 
thus
\begin{align*}
\ensuremath{\vec{H}}_\mathrm{s} = \vec{\zeta}_\mathrm{s}\ensuremath{ \mathbf{i} }_\mathrm{s}.
\end{align*}
We start by deriving the coupling equation \cite{Zhou_2008aa}  with the
definition of voltage as
\begin{align*}
\ensuremath{ \mathbf{v} }_{\mathrm{s}} = -\int_{\Omega}\vec{\chi}_\mathrm{s} \cdot \ensuremath{\vec{E}} \, \mathrm{d}\Omega.
\end{align*}
Using Gauss's theorem and Faraday-Lenz's law \eqref{eq:faraday}, we obtain
\begin{align*}
\ensuremath{ \mathbf{v} }_{\mathrm{s}} = \frac{\mathrm{d}}{\mathrm{d}t}\int_{\Omega}\vec{\zeta}_\mathrm{s}\cdot \ensuremath{\vec{B}} \, 
\mathrm{d}\Omega - 
\int_{\partial \Omega} (\vec{\zeta}_\mathrm{s}\times \ensuremath{\vec{E}}) \cdot \mathrm{d}\vec{S},
\end{align*}
which, due to the electric boundary conditions leads to the coupling equation
\begin{align*}
\ensuremath{ \mathbf{v} }_{\mathrm{s}} = \int_{\Omega}\vec{\zeta}_\mathrm{s} \cdot \frac{\mathrm{d}}{\mathrm{d}t}\mu(\vec{T} - \nabla\psi + 
\vec{\zeta}_\mathrm{s}\ensuremath{ \mathbf{i} }_\mathrm{s})  \, 
\mathrm{d}\Omega.
\end{align*} 

In the A* formulation, the voltage of the circuit can be related to the field quantities (see \cite{Schops_2013aa}) via 
\begin{align*}
\ensuremath{ \mathbf{v} }_{\mathrm{s}} =	\frac{\mathrm{d}}{\mathrm{d}t}\int_{\Omega}\vec{\chi}_\mathrm{s} \cdot \ensuremath{\vec{A}} \, 
\mathrm{d}\Omega.
\end{align*}

The degrees of freedom of both formulations are on dual sides of the diagram in Figure \ref{fig:maxhouse} and therefore it is said that both 
formulations are complementary.  Those type of 
systems can be used for an error approximation of the discretisation method 
\cite{Albanese_1991aa}. 
\begin{figure}[h]
	\centering
	\includegraphics[scale=1]{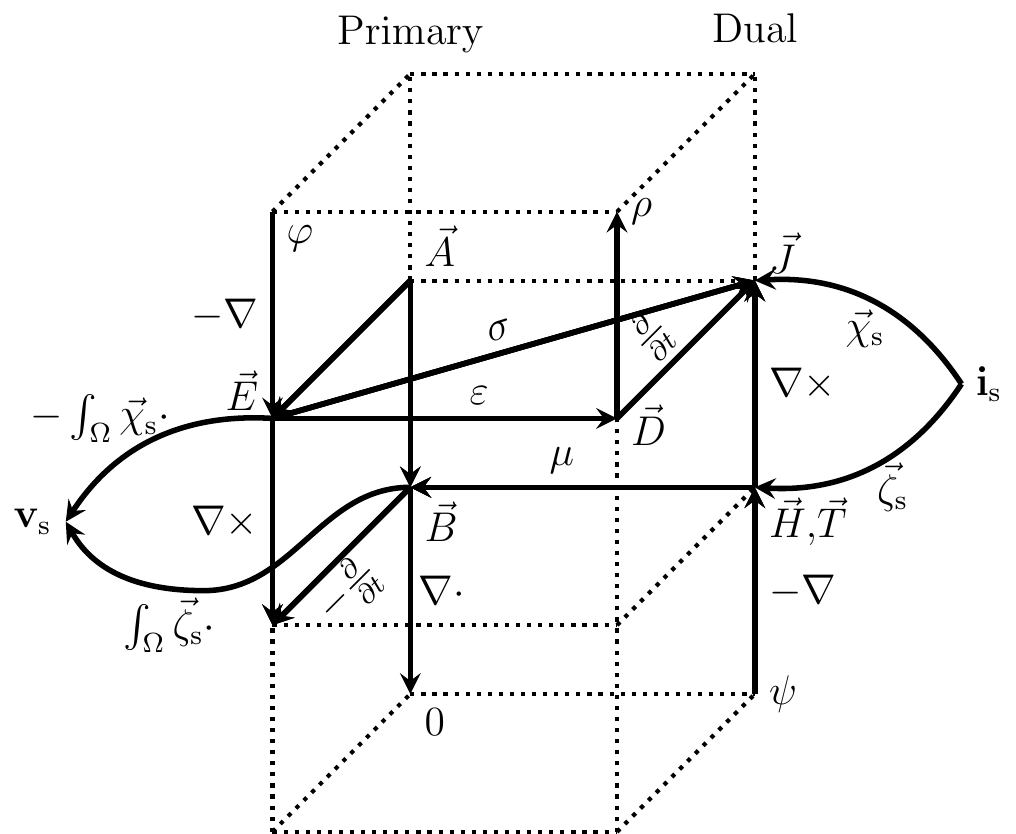}
	\caption{Maxwell House diagram, based on \cite{Schops_2013aa}.}\label{fig:maxhouse}
\end{figure}

\subsection{Discretised Systems}\label{sec:discreteSystems}
First the field model is discretised in space with a suitable method, such as the Finite Element Method (FEM) \cite{Monk_1993aa} with appropriate 
basis functions 
that 
fulfil the exact discrete de Rham 
sequence \cite{Bossavit_1998aa} or the Finite Integration Technique (FIT) \cite{Weiland_1977aa}. 

For the Finite Element discretisation with H(curl)-conforming basis functions $\vec{\nu}_i:\Omega\rightarrow\mathbbm{R}^3$, 
($i=1,\ldots,n_\mathrm{dof}$), 
the weak formulation in the 
case 
of 
the magnetic vector potential is \cite{Alonso-Rodriguez_2010aa}
\begin{align*}
\int_{\Omega}\sigma \frac{\partial \ensuremath{\vec{A}}}{\partial t}\cdot \vec{\nu}_i + \Big(\nu\nabla 
\times\ensuremath{\vec{A}}\Big)\cdot\Big(\nabla \times\vec{\nu}_i\Big) \, 
\mathrm{d}\Omega &= 
\int_{\Omega}\vec{\chi}_\mathrm{s}\ensuremath{ \mathbf{i} }_{\mathrm{s}}\cdot\vec{\nu}_i \, \mathrm{d}\Omega \\
\frac{\mathrm{d}}{\mathrm{d}t}\int_{\Omega}\vec{\chi}_\mathrm{s} \cdot \ensuremath{\vec{A}} \, \mathrm{d}\Omega &= \ensuremath{ \mathbf{v} 
}_{\mathrm{s}},
\end{align*}
for all $i$.
With the Ritz-Galerkin method the magnetic vector potential is approximated  by 
$$ \ensuremath{\vec{A}} \dot{=} \sum_{i=1}^{n_\mathrm{dof}}a_i(t)\vec{\nu}_i $$ 
and the conductivity matrix is for example constructed as
\begin{align}\label{eq:Msigma}
(\ensuremath{\mathbf{M}_{\sigma}})_{i,j} = \int_{\Omega}\sigma\vec{\nu}_i\cdot\vec{\nu}_j \, \mathrm{d}\Omega.
\end{align}
The T-$\Omega$ weak formulation as well as the rest of the material matrices are built analogously with the appropriate basis functions.

Eventually, the semi-discrete 
T-$\Omega$ formulation
\begin{align}\label{eq:tomega}
\ensuremath{\mathbf{C}}^{\top}\ensuremath{ \mathbf{M} }_\rho\ensuremath{\mathbf{C}}\mathbf{t} + 
\ensuremath{\mathbf{M}_{\mu}}\frac{\mathrm{d}}{\mathrm{d}t}\mathbf{t} + 
\ensuremath{\mathbf{M}_{\mu}}\ensuremath{\widetilde{\mathbf{S}}}^{\top}\frac{\mathrm{d}}{\mathrm{d}t}\Psi + 
\ensuremath{\mathbf{M}_{\mu}}\mathbf{Y}_\mathrm{s}\frac{\mathrm{d}}{\mathrm{d}t}\ensuremath{ \mathbf{i} }_\mathrm{s} & = 0 \nonumber \\
\ensuremath{\widetilde{\mathbf{S}}}\ensuremath{\mathbf{M}_{\mu}}\mathbf{t} + 
\ensuremath{\widetilde{\mathbf{S}}}\ensuremath{\mathbf{M}_{\mu}}\ensuremath{\widetilde{\mathbf{S}}}^{\top} \Psi + 
\ensuremath{\widetilde{\mathbf{S}}}\ensuremath{\mathbf{M}_{\mu}}\mathbf{Y}_\mathrm{s}\ensuremath{ \mathbf{i} }_\mathrm{s} &=0\\
\mathbf{Y}_\mathrm{s}^{\top}\ensuremath{\mathbf{M}_{\mu}}\frac{\mathrm{d}}{\mathrm{d}}\mathrm{t} + 
\mathbf{Y}_\mathrm{s}^{\top}\ensuremath{\mathbf{M}_{\mu}}\ensuremath{\widetilde{\mathbf{S}}}^{\top}\frac{\mathrm{d}}{\mathrm{d}t}\Psi + 
\mathbf{Y}_\mathrm{s}^{\top}\ensuremath{\mathbf{M}_{\mu}}\mathbf{Y}_\mathrm{s}\frac{\mathrm{d}}{\mathrm{d}t}\ensuremath{ \mathbf{i} }_\mathrm{s}  - 
\ensuremath{ \mathbf{v} }_{\mathrm{s}} &=0 
\nonumber
\end{align}
and the A* formulation
\begin{equation}\label{eq:astar}
\begin{aligned}
\ensuremath{\mathbf{M}_{\sigma}}\frac{\mathrm{d}}{\mathrm{d}t}\ensuremath{ \mathbf{a} } + 
\ensuremath{\mathbf{C}}^{\top}\ensuremath{\mathbf{M}_{\nu}}\ensuremath{\mathbf{C}}\ensuremath{ \mathbf{a} } - \ensuremath{ \mathbf{X} 
}_\mathrm{s}\ensuremath{ \mathbf{i} }_\mathrm{s} &=0\\
\frac{\mathrm{d}}{\mathrm{d}t}\ensuremath{ \mathbf{X} }_\mathrm{s}^{\top}\ensuremath{ \mathbf{a} } - \ensuremath{ \mathbf{v} }_{\mathrm{s}} &= 0,
\end{aligned}
\end{equation}
lead to two different systems of DAEs describing the same physical phenomenon.
Here, $\ensuremath{ \mathbf{M} }_\star$ are the material matrices that describe the material relations between the discrete quantities, 
$\ensuremath{\mathbf{C}}$, $-\ensuremath{\widetilde{\mathbf{S}}}^{\top}$ 
and  
$\ensuremath{\widetilde{\mathbf{S}}}$ are the discrete curl, gradient and divergence operators. $\ensuremath{ \mathbf{X} }_\mathrm{s}$ and 
$\mathbf{Y}_\mathrm{s}$  are the discretisations of the winding functions $\vec{\chi}_\mathrm{s}$ and  $\vec{\zeta}_\mathrm{s}$ respectively. The in 
the FEM notation not very common 
matrix factorisation of systems \eqref{eq:tomega} and \eqref{eq:astar} borrowed from the Finite Integration Technique \cite{Weiland_1996aa}  is  
used for convenience in the analysis below.

In order to solve both systems, consistent initial conditions are imposed for $\ensuremath{ \mathbf{a} }(t_0) = \ensuremath{ \mathbf{a} }_0$, 
$\mathbf{t}(t_0) = \mathbf{t}_0$, $\Psi(t_0) = 
\Psi_0$ and 
either $\ensuremath{ \mathbf{v} }_{\mathrm{s}}(t_0)=\ensuremath{ \mathbf{v} }_{\mathrm{s},0}$ or $\ensuremath{ \mathbf{i} 
}_{\mathrm{s}}(t_0)=\ensuremath{ \mathbf{i} }_{\mathrm{s},0}$, depending on which lumped quantity is 
given as an excitation to the system. To ensure uniqueness of solution, also discrete gauging conditions are inserted (see e.g. 
\cite{Clemens_2002aa,Rubinacci_1988aa}).

\begin{prop}[System matrices]\label{property:matrices}
	The discretisation matrices have the following properties.
	\begin{itemize}
		\item 	The material matrices $\ensuremath{ \mathbf{M} }_\star$ are symmetric positive definite for $\star = \{\mu, \nu\}$ and symmetric 
		positive semidefinite for 
		$\star = \{\sigma\}$. 
		\item The discrete gradient matrix $-\ensuremath{\widetilde{\mathbf{S}}}^{\top}$ is assumed to be projected to a subspace where the boundary 
		conditions are imposed to the 
		degrees of freedom and thus has full column rank, that is, $\ker\ensuremath{\widetilde{\mathbf{S}}}^{\top} = \{\mathbf{0}\}$.
		\item For the mentioned suitable discretisations, $\ensuremath{\mathbf{C}}\ensuremath{\widetilde{\mathbf{S}}}^{\top} = \mathbf{0}$.
	\end{itemize}
\end{prop}
For FIT all three properties are classical results (see e.g. \cite{Weiland_1996aa}). In the case of FEM, the first property follows 
directly from how the material matrices are constructed analogously to equation \eqref{eq:Msigma}. Both the second and third properties follows from 
the fact that the spaces spanned by the basis functions including boundary conditions fulfil the de Rham sequence.

\section{DAE Index Analysis}\label{sec:index}
Before heading to the index results of both field-circuit coupled systems, a 
specific structured inductance-like element, that eases the later analysis, is introduced.
\begin{prop}[Inductance-like element]\label{prop:induct-like}
	A device described by a DAE
	\begin{equation*}
	\ensuremath{ \mathbf{F} }\left(\frac{\mathrm{d}}{\mathrm{d}t}\ensuremath{ \mathbf{x} }_\lambda, \frac{\mathrm{d}}{\mathrm{d}t}\ensuremath{ 
	\mathbf{i} }_\lambda, \ensuremath{ \mathbf{x} }_\lambda, \ensuremath{ \mathbf{i} }_\lambda, \ensuremath{ \mathbf{v} }_\lambda, t\right) = 0, 
	\end{equation*}
	where at most one differentiation $\frac{\mathrm{d}}{\mathrm{d}t}\ensuremath{ \mathbf{F} }\left(\frac{\mathrm{d}}{\mathrm{d}t}\ensuremath{ 
	\mathbf{x} }_\lambda, \frac{\mathrm{d}}{\mathrm{d}t}\ensuremath{ \mathbf{i} }_\lambda, \ensuremath{ \mathbf{x} }_\lambda, 
	\ensuremath{ \mathbf{i} }_\lambda, \ensuremath{ \mathbf{v} }_\lambda, t\right) = 0$ is needed, such that one can write
	\begin{align}
	\frac{\mathrm{d}}{\mathrm{d}t}\ensuremath{ \mathbf{x} }_\lambda &= \mathbf{f}_{\ensuremath{ \mathbf{x} }}(\ensuremath{ \mathbf{x} }_\lambda, 
	\ensuremath{ \mathbf{i} }_\lambda, \ensuremath{ \mathbf{v} }_\lambda, t)\\
	\frac{\mathrm{d}}{\mathrm{d}t}\ensuremath{ \mathbf{i} }_\lambda &= \ensuremath{ \mathbf{L} }_\lambda(\ensuremath{ \mathbf{x} 
	}_\lambda)^{-1}\ensuremath{ \mathbf{v} }_\lambda + \mathbf{f}_{\ensuremath{ \mathbf{i} }}(\ensuremath{ \mathbf{x} }_\lambda, 
	\ensuremath{ \mathbf{i} }_\lambda, t),\label{eq:ilambda}
	\end{align}
	with $\ensuremath{ \mathbf{L} }_\lambda(\ensuremath{ \mathbf{x} }_\lambda)$ being positive definite, is an inductance-like device. 
\end{prop}

\begin{proof}
	We define $\phi(\ensuremath{ \mathbf{i} }_\lambda, \ensuremath{ \mathbf{x} }_\lambda, t) = \ensuremath{ \mathbf{L} }_{\lambda}(\ensuremath{ 
	\mathbf{x} }_\lambda)\ensuremath{ \mathbf{i} }_\lambda$. The first property is thus fulfilled, as
	\begin{equation*}
	\frac{\partial}{\partial \ensuremath{ \mathbf{i} }_\lambda}\phi(\ensuremath{ \mathbf{x} }_\lambda, \ensuremath{ \mathbf{i} }_\lambda, t) = 
	\ensuremath{ \mathbf{L} }_\lambda(\ensuremath{ \mathbf{x} }_\lambda),
	\end{equation*}
	which is positive definite and therefore regular. Also, setting
	\begin{align*}
	\mathbf{{f}}_{\phi}(\ensuremath{ \mathbf{x} }_\lambda, \ensuremath{ \mathbf{i} }_\lambda, \ensuremath{ \mathbf{v} }_\lambda, t)= 
	\ensuremath{ \mathbf{v} }_\lambda + \frac{\partial}{\partial \ensuremath{ \mathbf{x} }_\lambda}\big(\ensuremath{ \mathbf{L} 
	}_\lambda(\ensuremath{ \mathbf{x} }_\lambda)\ensuremath{ \mathbf{i} }_\lambda\big)
	\mathbf{f}_{\ensuremath{ \mathbf{x} }}(\ensuremath{ \mathbf{x} }_\lambda, \ensuremath{ \mathbf{i} }_\lambda, \ensuremath{ \mathbf{v} }_\lambda, 
	t) + 
	\ensuremath{ \mathbf{L} }_\lambda(\ensuremath{ \mathbf{x} }_\lambda)\mathbf{f}_{\ensuremath{ \mathbf{i} }}(\ensuremath{ \mathbf{x} }_\lambda, 
	\ensuremath{ \mathbf{i} }_\lambda, t)
	\end{align*}
	leads to equation \eqref{eq:ilambda}, where $\frac{\partial}{\partial 
		\ensuremath{ \mathbf{v} }_\lambda}\Big(\left(\frac{\partial \phi}{\partial \ensuremath{ \mathbf{i} }_\lambda}\right)^{-1}\left( 
	-\frac{\partial \phi}{\partial 
		\ensuremath{ \mathbf{x} }_\lambda}\mathbf{f}_{\ensuremath{ \mathbf{x} }} -  \frac{\partial \phi}{\partial t} + \mathbf{{f}}_{\phi}\right)
	\Big) = \ensuremath{ \mathbf{L} }_{\lambda}^{-1}(\ensuremath{ \mathbf{x} }_\lambda)$ is again positive definite and fulfils the second property 
	of an 
	inductance-like element. 
\end{proof}

\subsection{DAE Index of the T-$\Omega$  Formulation}\label{sec:T-Omega}
Let the tree-cotree gauge \cite{Albanese_1988aa}
be introduced. For a simply connected region $\Omega_\mathrm{c}$, the values of the degrees of freedom $\mathbf{t}$ are set to zero on a tree 
$T_\mathrm{c}$ of the mesh 
inside the conducting 
region $\Omega_\mathrm{c}$ that adequately takes care of the boundary conditions.
For that, a projector $\bar{\ensuremath{ \mathbf{P} }}$ is defined that projects onto the edges of the cotree of $T_\mathrm{c}$. $\ensuremath{ 
\mathbf{P} }$ is the reduction of the 
projection matrix $\bar{\ensuremath{ \mathbf{P} }}$ by deleting all the zero columns. In case of a multiply-connected region, cuts have to be defined 
in 
$\Omega_\mathrm{c}$  to ensure a correct gauging condition (see \cite{Zhou_2008aa}).
\begin{assumption}[Gauged T-$\Omega$ formulation]\label{ass:treecotree}
	The discretised T-$\Omega$ system \eqref{eq:tomega} is gauged and thus rewritten as
	\begin{subequations}\label{eq:tomegauge}
		\begin{align}
		\ensuremath{ \mathbf{P} }^{\top}\ensuremath{\mathbf{C}}^{\top}\ensuremath{ \mathbf{M} }_\rho\ensuremath{\mathbf{C}} \ensuremath{ \mathbf{P} 
		}\mathbf{t}_\mathrm{red} + 
		\ensuremath{ \mathbf{P} }^{\top}\ensuremath{\mathbf{M}_{\mu}}\left(\ensuremath{ \mathbf{P} 
		}\frac{\mathrm{d}}{\mathrm{d}t}\mathbf{t}_\mathrm{red} + 
		\ensuremath{\widetilde{\mathbf{S}}}^{\top}\frac{\mathrm{d}}{\mathrm{d}t}\Psi + 
		\mathbf{Y}_\mathrm{s}\frac{\mathrm{d}}{\mathrm{d}t}\ensuremath{ \mathbf{i} }_\mathrm{s}\right) & = 0\label{eq:tomega1} \\
		\ensuremath{\widetilde{\mathbf{S}}}\ensuremath{\mathbf{M}_{\mu}}\left(\ensuremath{ \mathbf{P} }\mathbf{t}_\mathrm{red} + 
		\ensuremath{\widetilde{\mathbf{S}}}^{\top} \Psi + \mathbf{Y}_\mathrm{s}\ensuremath{ \mathbf{i} }_\mathrm{s}\right) &=0\label{eq:tomega2}\\
		\mathbf{Y}_\mathrm{s}^{\top}\ensuremath{\mathbf{M}_{\mu}}\left(\ensuremath{ \mathbf{P} }\frac{\mathrm{d}}{\mathrm{d}t}\mathbf{t}_\mathrm{red} 
		+ \ensuremath{\widetilde{\mathbf{S}}}^{\top}\frac{\mathrm{d}}{\mathrm{d}t}\Psi + 
		\mathbf{Y}_\mathrm{s}\frac{\mathrm{d}}{\mathrm{d}t}\ensuremath{ \mathbf{i} }_\mathrm{s} \right) - \ensuremath{ \mathbf{v} }_{\mathrm{s}} 
		&=0.\label{eq:tomega3} 
		\end{align}
	\end{subequations}
	such that
	the matrix $\ensuremath{ \mathbf{K} }_\rho = \ensuremath{ \mathbf{P} }^{\top}\ensuremath{\mathbf{C}}^{\top}\ensuremath{ \mathbf{M} 
	}_\rho\ensuremath{\mathbf{C}} \ensuremath{ \mathbf{P} }$ has full rank, i.e. $\det(\ensuremath{ \mathbf{K} }_\rho) \neq 0$.
\end{assumption}
\begin{prop}\label{prop:PS}
	The discrete field $\ensuremath{ \mathbf{P} }\ensuremath{ \mathbf{x} }$ is not a gradient field, i.e. $\ensuremath{ \mathbf{P} }\ensuremath{ 
	\mathbf{x} }\neq \ensuremath{\widetilde{\mathbf{S}}}^{\top}\mathbf{y}$ for $\ensuremath{ \mathbf{x} }, \mathbf{y}\neq\mathbf{0}$.
\end{prop}
\begin{proof}
	The Proposition follows directly from Assumption \ref{ass:treecotree} and Property \ref{property:matrices}.
\end{proof}

\begin{prop}[Discrete Helmholtz split]\label{prop:decomp}
	Every $\ensuremath{ \mathbf{x} } \in \mathbbm{R}^n$ can be written as $\ensuremath{ \mathbf{x} } = \ensuremath{\widetilde{\mathbf{S}}}^{\top} 
	\ensuremath{ \mathbf{x} }_1 + \ensuremath{\mathbf{M}_{\mu}}^{-1}\ensuremath{\mathbf{C}}^{\top}\ensuremath{ \mathbf{x} }_2$, where 
	$\ensuremath{ \mathbf{x} }_1 \in \mathbbm{R}^m$, $\ensuremath{ \mathbf{x} }_2\in\mathbbm{R}^{n}$, with 
	$\ensuremath{\widetilde{\mathbf{S}}}^{\top}\in\mathbbm{R}^{n\times m}$ and $\ensuremath{\mathbf{C}} \in \mathbbm{R}^{n\times n}$ being the 
	matrices 
	defined in Section \ref{sec:discreteSystems}. 
\end{prop}
\begin{proof}
	As $\ensuremath{ \mathbf{M} }_\mu$ is positive definite and  
	$\ker(\ensuremath{\widetilde{\mathbf{S}}}\ensuremath{\mathbf{M}_{\mu}}^{\frac{1}{2}}$)$= 
	\mathrm{im}$($\ensuremath{\mathbf{M}_{\mu}}^{-\frac{1}{2}}\ensuremath{\mathbf{C}}^{\top}$), we have
	\begin{align*}
	\mathbf{y} = \ensuremath{\mathbf{M}_{\mu}}^{\frac{1}{2}}\ensuremath{\widetilde{\mathbf{S}}}^{\top}\mathbf{y}_1 + 
	\ensuremath{\mathbf{M}_{\mu}}^{-\frac{1}{2}}\ensuremath{\mathbf{C}}^{\top}\mathbf{y}_2,
	\end{align*}
	for all $\mathbf{y}\in \mathbbm{R}^n$. Furthermore, there exists  a
	$\mathbf{y}_0$ such that $\ensuremath{ \mathbf{x} } = \ensuremath{\mathbf{M}_{\mu}}^{-\frac{1}{2}}\mathbf{y}_0 = 
	\ensuremath{\widetilde{\mathbf{S}}}^{\top} \ensuremath{ \mathbf{x} }_1 + 
	\ensuremath{\mathbf{M}_{\mu}}^{-1}\ensuremath{\mathbf{C}}^{\top}\ensuremath{ \mathbf{x} }_2$.
\end{proof}

\begin{assumption}[Discrete current densities]\label{ass:curlYP}
	It is assumed that
	\begin{align*}
	\mathbf{0}\neq\ensuremath{\mathbf{C}}\mathbf{Y}_\mathrm{s}\mathbf{x} \neq \ensuremath{\mathbf{C}}\ensuremath{ \mathbf{P} }\mathbf{y}, \text{ for 
	}\ensuremath{ \mathbf{x} },\mathbf{y}\neq \mathbf{0},
	\end{align*}
	where we recall that $\mathbf{Y}_\mathrm{s}$ is the discrete winding function of the T-$\Omega$ formulation.
\end{assumption}
The previous assumption imposes that the curl of the discretised magnetic source field, which is the discretised source current density 
$\mathbf{j}^{(r)}_\mathrm{s}$ associated with $\Omega^{(r)}_\mathrm{s}$, is different from the curl of the discretised electric vector potential 
$\mathbf{t}_\mathrm{red}$, which corresponds to the conduction current density $\mathbf{j}_\mathrm{c}$ associated with  $\Omega_\mathrm{c}$. As 
$\Omega_\mathrm{s}^{(i)}\cap\Omega_\mathrm{c}=\emptyset$ for all $i$, the assumption is reasonable. 

\begin{prop}[T-$\Omega$ inductance-like element]\label{prop:tomega}
	The discrete (gauged) system of equations of the T-$\Omega$ formulation with circuit coupling equation \eqref{eq:tomegauge}  is an 
	inductance-like element. 
\end{prop}
\begin{proof}
	By differentiating equation \eqref{eq:tomega2} once, one can extract
	\begin{equation}\label{eq:proofpsi}
	\frac{\mathrm{d}}{\mathrm{d}t}\Psi = -\ensuremath{ \mathbf{L} 
	}_\mu^{-1}\ensuremath{\widetilde{\mathbf{S}}}\ensuremath{\mathbf{M}_{\mu}}\ensuremath{ \mathbf{P} 
	}\frac{\mathrm{d}}{\mathrm{d}t}\mathbf{t}_\mathrm{red} 
	- \ensuremath{ \mathbf{L} 
	}_\mu^{-1}\ensuremath{\widetilde{\mathbf{S}}}\ensuremath{\mathbf{M}_{\mu}}\mathbf{Y}_\mathrm{s}\frac{\mathrm{d}}{\mathrm{d}t}\ensuremath{ 
	\mathbf{i} }_{\mathrm{s}},
	\end{equation}
	with $\ensuremath{ \mathbf{L} }_\mu = \ensuremath{\widetilde{\mathbf{S}}}\ensuremath{\mathbf{M}_{\mu}}\ensuremath{\widetilde{\mathbf{S}}}^{\top}$ 
	positive definite due to Property \ref{property:matrices}. 
	
	In order to obtain an expression $\frac{\mathrm{d}}{\mathrm{d}t}\mathbf{t}_\mathrm{red}$, first the positive definiteness
	of the matrix $\ensuremath{ \mathbf{P} }^{\top} \mathbf{W}\ensuremath{ \mathbf{P} }$, with $\mathbf{W} = \ensuremath{\mathbf{M}_{\mu}} - 
	\ensuremath{\mathbf{M}_{\mu}}\ensuremath{\widetilde{\mathbf{S}}}^{\top}\ensuremath{ \mathbf{L} 
	}_\mu^{-1}\ensuremath{\widetilde{\mathbf{S}}}\ensuremath{\mathbf{M}_{\mu}}$ is shown.
	\begin{align*}
	\mathbf{W} =	\ensuremath{\mathbf{M}_{\mu}}^{\frac{1}{2}}(\ensuremath{ \mathbf{I} } - 
	\ensuremath{\mathbf{M}_{\mu}}^{\frac{1}{2}}\ensuremath{\widetilde{\mathbf{S}}}^{\top}\ensuremath{ \mathbf{L} 
	}_\mu^{-1}\ensuremath{\widetilde{\mathbf{S}}}\ensuremath{\mathbf{M}_{\mu}}^{\frac{1}{2}})\ensuremath{\mathbf{M}_{\mu}}^{\frac{1}{2}} = 
	\ensuremath{\mathbf{M}_{\mu}}^{\frac{1}{2}}\ensuremath{ \mathbf{Q} }_\mu\ensuremath{\mathbf{M}_{\mu}}^{\frac{1}{2}},
	\end{align*}
	where $\ensuremath{ \mathbf{Q} }_\mu$ is a projector and thus positive semidefinite. 
	Let us assume there exists an $\ensuremath{ \mathbf{x} }\neq0$ with $\ensuremath{ \mathbf{x} }^{\top}\ensuremath{ \mathbf{P} }^{\top} 
	\mathbf{W}\ensuremath{ \mathbf{P} }\ensuremath{ \mathbf{x} } = 0$. Then $ 
	\mathbf{W}^{\frac{1}{2}}\ensuremath{ \mathbf{P} }\ensuremath{ \mathbf{x} } = \mathbf{0}$ and $\mathbf{W}\ensuremath{ \mathbf{P} }\ensuremath{ 
	\mathbf{x} } = \mathbf{0}$ and thus
	$\ensuremath{ \mathbf{P} }\ensuremath{ \mathbf{x} } = \ensuremath{\widetilde{\mathbf{S}}}^{\top}\ensuremath{ \mathbf{L} 
	}_\mu^{-1}\ensuremath{\widetilde{\mathbf{S}}}\ensuremath{\mathbf{M}_{\mu}}\ensuremath{ \mathbf{P} }\ensuremath{ \mathbf{x} }$, which cannot be 
	due to Proposition \ref{prop:PS}.
	
	Secondly, inserting \eqref{eq:proofpsi} into 
	\eqref{eq:tomega1}, yields 
	\begin{equation}\label{eq:prooftred}
	\frac{\mathrm{d}}{\mathrm{d}t}\mathbf{t}_\mathrm{red} = -(\ensuremath{ \mathbf{P} }^{\top} 
	\mathbf{W}\ensuremath{ \mathbf{P} })^{-1}\ensuremath{ \mathbf{P} 
	}_\sigma^{\top}\mathbf{W}\mathbf{Y}_\mathrm{s}\frac{\mathrm{d}}{\mathrm{d}t}\ensuremath{ \mathbf{i} }_{\mathrm{s}}
	-(\ensuremath{ \mathbf{P} }^{\top} \mathbf{W}\ensuremath{ \mathbf{P} })^{-1}\ensuremath{ \mathbf{P} }^{\top}\ensuremath{ \mathbf{K} 
	}_\rho\ensuremath{ \mathbf{P} }\mathbf{t}_\mathrm{red}.
	\end{equation}
	Finally, using equations \eqref{eq:proofpsi}, \eqref{eq:prooftred} and \eqref{eq:as3}, one obtains
	\begin{equation}
	\ensuremath{ \mathbf{v} }_{\mathrm{s}} = \ensuremath{ \mathbf{L} }_\lambda\frac{\mathrm{d}}{\mathrm{d}t}\ensuremath{ \mathbf{i} }_{\mathrm{s}} 
	-\mathbf{Y}_\mathrm{s}^{\top}\mathbf{W}\ensuremath{ \mathbf{P} }(\ensuremath{ \mathbf{P} }^{\top}\mathbf{W}\ensuremath{ \mathbf{P} })^{-1}
	\mathbf{K}_\rho\mathbf{t}_\mathrm{red},
	\end{equation}
	with 
	\begin{equation*}
	\ensuremath{ \mathbf{L} }_\lambda = \mathbf{Y}_\mathrm{s}^{\top}(\mathbf{W} - 
	\mathbf{W}\ensuremath{ \mathbf{P} }(\ensuremath{ \mathbf{P} }^{\top}\mathbf{W}\ensuremath{ \mathbf{P} })^{-1}\ensuremath{ \mathbf{P} 
	}^{\top}\mathbf{W})\mathbf{Y}_\mathrm{s} 
	:=\mathbf{Y}_\mathrm{s}^{\top}\mathbf{W}_\mathrm{P}\mathbf{Y}_\mathrm{s}.
	\end{equation*}
	If $\ensuremath{ \mathbf{L} }_\lambda$ is positive 
	definite, then, 
	applying Proposition \ref{prop:induct-like} concludes the proof. 
	
	\vspace{0.5em}
	\noindent Let us verify that $\ensuremath{ \mathbf{x} }^{\top}\ensuremath{ \mathbf{L} }_{\lambda}\ensuremath{ \mathbf{x} } > 0$, for 
	$\ensuremath{ \mathbf{x} }\neq 
	\mathbf{0}$. Analogously to the case of $\ensuremath{ \mathbf{P} }^{\top}\mathbf{W}\ensuremath{ \mathbf{P} }$, one can show that 
	$\mathbf{W}_\mathrm{P} = 
	\mathbf{W}^{\frac{1}{2}}\ensuremath{ \mathbf{Q} }_{\mathrm{W}}\mathbf{W}^{\frac{1}{2}}$, with $\ensuremath{ \mathbf{Q} }_{\mathrm{W}}$ being a 
	projector and again we just need to 
	show that $(\mathbf{W} - 
	\mathbf{W}\ensuremath{ \mathbf{P} }(\ensuremath{ \mathbf{P} }^{\top}\mathbf{W}\ensuremath{ \mathbf{P} })^{-1}\ensuremath{ \mathbf{P} 
	}^{\top}\mathbf{W})\mathbf{Y}_\mathrm{s}\ensuremath{ \mathbf{x} }\neq \mathbf{0}$. Let us assume 
	$(\mathbf{W} - 
	\mathbf{W}\ensuremath{ \mathbf{P} }(\ensuremath{ \mathbf{P} }^{\top}\mathbf{W}\ensuremath{ \mathbf{P} })^{-1}\ensuremath{ \mathbf{P} 
	}^{\top}\mathbf{W})\mathbf{Y}_\mathrm{s}\ensuremath{ \mathbf{x} } = \mathbf{0}$, then
	\begin{equation*}
	\mathbf{W}\mathbf{Y}_\mathrm{s}\ensuremath{ \mathbf{x} } = \mathbf{W}\ensuremath{ \mathbf{P} }(\ensuremath{ \mathbf{P} 
	}^{\top}\mathbf{W}\ensuremath{ \mathbf{P} })^{-1}\ensuremath{ \mathbf{P} }^{\top}\mathbf{W}\mathbf{Y}_\mathrm{s}\ensuremath{ \mathbf{x} }=:
	\mathbf{W}\ensuremath{ \mathbf{P} }\mathbf{y}.
	\end{equation*}
	Using Proposition \ref{prop:decomp}, we can write
	\begin{align*}
	\mathbf{Y}_\mathrm{s}\ensuremath{ \mathbf{x} } = \ensuremath{\widetilde{\mathbf{S}}}^{\top}\ensuremath{ \mathbf{x} }_1 + 
	\ensuremath{\mathbf{M}_{\mu}}^{-1}\ensuremath{\mathbf{C}}^{\top}\ensuremath{ \mathbf{x} }_2 && \text{and} &&
	\ensuremath{ \mathbf{P} }\mathbf{y} = \ensuremath{\widetilde{\mathbf{S}}}^{\top}\mathbf{y}_1 + 
	\ensuremath{\mathbf{M}_{\mu}}^{-1}\ensuremath{\mathbf{C}}^{\top}\mathbf{y}_2
	\end{align*}
	and
	\begin{align*}
	\mathbf{W}\mathbf{Y}_\mathrm{s}\ensuremath{ \mathbf{x} } = \ensuremath{\mathbf{C}}^{\top}\ensuremath{ \mathbf{x} }_2 = \mathbf{W}\ensuremath{ 
	\mathbf{P} }\mathbf{y} = \ensuremath{\mathbf{C}}^{\top}\mathbf{y}_2.
	\end{align*}
	But, according to Assumption \ref{ass:curlYP},
	\begin{align*}
	\ensuremath{\mathbf{C}}\ensuremath{\mathbf{M}_{\mu}}^{-1}\ensuremath{\mathbf{C}}^{\top}\ensuremath{ \mathbf{x} }_2 = 
	\ensuremath{\mathbf{C}}\mathbf{Y}_\mathrm{s}\ensuremath{ \mathbf{x} } \neq \ensuremath{\mathbf{C}}\ensuremath{ \mathbf{P} }\mathrm{y} 
	= \ensuremath{\mathbf{C}}\ensuremath{\mathbf{M}_{\mu}}^{-1}\ensuremath{\mathbf{C}}^{\top}\mathbf{y}_2,
	\end{align*}
	thus $ \ensuremath{\mathbf{C}}^{\top}\ensuremath{ \mathbf{x} }_2 \neq \ensuremath{\mathbf{C}}^{\top}\mathbf{y}_2$ and 
	$\mathbf{W}\mathbf{Y}_\mathrm{s}\ensuremath{ \mathbf{x} } \neq  \mathbf{W}\ensuremath{ \mathbf{P} }\mathbf{y}$. 
	Therefore, 
	$\ensuremath{ \mathbf{L} }_\lambda$ is positive definite which concludes the proof.
\end{proof}
We have proven that the T-$\Omega$ formulation embedded in a circuit behaves like an inductance  from the index point of view. Exciting the field 
model either with a current source or a voltage source can be seen as a particular circuit coupling and this yields the following Corollary.
\begin{corollary}[Excitation index of the T-$\Omega$ formulation]
	The discrete (gauged) system of equations of the T-$\Omega$ formulation with circuit coupling equation \eqref{eq:tomegauge} has differentiation 
	index
	\begin{itemize}
		\item 1, if the voltage $\ensuremath{ \mathbf{v} }_\mathrm{s} $ is prescribed.
		\item 2, if the current $\ensuremath{ \mathbf{i} }_\mathrm{s}$ is prescribed.
	\end{itemize}
\end{corollary}
Therefore, a voltage excitation leads to a system with a lower index and, hence, can be numerically handled in an easier way.
Now, the same analysis will be done for the A* formulation in order to compare both cases.

\subsection{DAE Index of the A* Formulation}

Again a tree-cotree gauge is introduced, albeit this time in the non-conducting domain $\Omega_\mathrm{c}^\mathsf{c}$. Notice that the reduction of 
the 
projection matrix $\ensuremath{ \mathbf{P} }$ that deletes the necessary degrees of freedom is different from the one in the T-$\Omega$ formulation 
\eqref{eq:tomegauge}.
\begin{assumption}[Gauged A* formulation]\label{ass:agauge}
	The discrete system \eqref{eq:astar} is gauged and thus rewritten as 
	\begin{equation}\label{eq:astargauge}
	\begin{aligned}
	\ensuremath{ \mathbf{P} }^{\top}\ensuremath{\mathbf{M}_{\sigma}}\ensuremath{ \mathbf{P} }\frac{\mathrm{d}}{\mathrm{d}t}\ensuremath{ \mathbf{a} 
	}_\mathrm{red} + \ensuremath{ \mathbf{P} }^{\top}\ensuremath{\mathbf{C}}^{\top}\ensuremath{\mathbf{M}_{\nu}}\ensuremath{\mathbf{C}}\ensuremath{ 
	\mathbf{P} } \ensuremath{ \mathbf{a} }_\mathrm{red} - 
	\ensuremath{ \mathbf{P} }^{\top}\ensuremath{ \mathbf{X} }_\mathrm{s}\ensuremath{ \mathbf{i} }_\mathrm{s} &=0\\
	\frac{\mathrm{d}}{\mathrm{d}t}\ensuremath{ \mathbf{X} }_\mathrm{s}^{\top}\ensuremath{ \mathbf{P} }\ensuremath{ \mathbf{a} }_\mathrm{red} - 
	\ensuremath{ \mathbf{v} }_{\mathrm{s}} &= 0,
	\end{aligned}
	\end{equation}
	such that the matrix 
	pencil $\lambda\mathbf{\bar{M}}_\sigma + \ensuremath{ \mathbf{K} }_\nu$ is positive definite for $\lambda>0$ with $\mathbf{\bar{M}}_\sigma = 
	\ensuremath{ \mathbf{P} }^{\top}\ensuremath{\mathbf{M}_{\sigma}}\ensuremath{ \mathbf{P} }$
	and $\ensuremath{ \mathbf{K} }_\nu = 
	\ensuremath{ \mathbf{P} }^{\top}\ensuremath{\mathbf{C}}^{\top}\ensuremath{\mathbf{M}_{\nu}}\ensuremath{\mathbf{C}}\ensuremath{ \mathbf{P} }$.
\end{assumption}
For simplicity of notation, we introduce the matrix $\mathbf{\bar{X}}_\mathrm{s} = \ensuremath{ \mathbf{P} }^{\top}\ensuremath{ \mathbf{X} 
}_\mathrm{s}$.
\begin{assumption}[Discrete winding function]\label{ass:discretedomain}
	The discrete gauged winding function matrix $\mathbf{\bar{X}}_\mathrm{s}$ fulfils
	\begin{itemize}
		\item it has full column rank.
		\item $\mathrm{im} \mathbf{\bar{X}}_\mathrm{s}\perp\mathrm{im}\mathbf{\bar{M}}_\sigma$.
	\end{itemize}
\end{assumption}
This last assumption states properties of the discrete matrices, inspired by the properties of the domains and the continuous functions stated in 
Assumption \ref{ass:domain}. 
The first property is motivated by the fact that each column of matrix $\mathbf{\bar{X}}_\mathrm{s}$ is the discretisation of a different winding 
function $\vec{\chi}_\mathrm{s}^{(j)}$ which all have disjoint supports.  Similarly, the conducting domain and the source domain are disjoint, which 
inspires the second property that could be relaxed but is kept for simplicity.

\begin{prop}[A* inductance-like element]\label{prop:astar}
	The discrete (gauged) system of equations of the A* formulation with coupling equation \eqref{eq:astargauge}  is an inductance-like element. 
\end{prop}

\begin{proof}
	First, a projector $\ensuremath{ \mathbf{Q} }_\sigma$ onto $\ker\mathbf{\bar{M}}_\sigma$ is defined and $\ensuremath{ \mathbf{P} }_\sigma = 
	\ensuremath{ \mathbf{I} }-\ensuremath{ \mathbf{Q} }_\sigma$. System 
	\eqref{eq:astargauge} 
	can 
	now be rewritten as
	\begin{subequations}\label{eq:astargauged}
		\begin{align}
		\bar{\ensuremath{ \mathbf{M} }}_\sigma\frac{\mathrm{d}}{\mathrm{d}t}\ensuremath{ \mathbf{a} }_\mathrm{red} + \ensuremath{ \mathbf{P} 
		}_\sigma^{\top} \ensuremath{ \mathbf{K} }_\nu\ensuremath{ \mathbf{a} }_\mathrm{red} - 
		\ensuremath{ \mathbf{P} }_\sigma^{\top}\mathbf{\bar{X}}_\mathrm{s}\ensuremath{ \mathbf{i} }_{\mathrm{s}} 
		&=0\label{eq:as1}\\
		\ensuremath{ \mathbf{Q} }_\sigma^{\top}\ensuremath{ \mathbf{K} }_\nu\ensuremath{ \mathbf{a} }_\mathrm{red} - \ensuremath{ \mathbf{Q} 
		}_\sigma^{\top}\mathbf{\bar{X}}_\mathrm{s}\ensuremath{ \mathbf{i} }_{\mathrm{s}} &=0\label{eq:as2}\\
		\frac{\mathrm{d}}{\mathrm{d}t}\mathbf{\bar{X}}_\mathrm{s}^{\top}\ensuremath{ \mathbf{a} }_\mathrm{red}  - \ensuremath{ \mathbf{v} 
		}_{\mathrm{s}} &=0.\label{eq:as3}
		\end{align}
	\end{subequations}
	Equation \eqref{eq:as1} allows to extract $\ensuremath{ \mathbf{P} }_\sigma\frac{\mathrm{d}}{\mathrm{d}t}\ensuremath{ \mathbf{a} }_\mathrm{red} = 
	\mathbf{f}_{\sigma}(\ensuremath{ \mathbf{a} }_\mathrm{red}, 
	\ensuremath{ \mathbf{i} }_{\mathrm{s}})$, where $\mathbf{f}_{\sigma}$ can directly be computed from the equation.
	After one time differentiation of \eqref{eq:as2} and inserting into \eqref{eq:as3}
	\begin{align*}
	\frac{\mathrm{d}}{\mathrm{d}t}\ensuremath{ \mathbf{i} }_{\mathrm{s}} = \ensuremath{ \mathbf{L} }_\lambda^{-1}\ensuremath{ \mathbf{v} 
	}_{\mathrm{s}} + \mathbf{f}_{\mathrm{s}}(\ensuremath{ \mathbf{a} }_\mathrm{red}, \ensuremath{ \mathbf{i} }_\mathrm{s}),
	\end{align*}
	is obtained, with $\mathbf{f}_{\mathrm{s}}$ being a result of inserting $\mathbf{f}_{\sigma}$ into \eqref{eq:as3}. Here $\ensuremath{ \mathbf{L} 
	}_\lambda = 
	\mathbf{\bar{X}}_\mathrm{s}^{\top}\ensuremath{ \mathbf{Q} }_\sigma(\ensuremath{ \mathbf{Q} }_\sigma^{\top}\ensuremath{ \mathbf{K} 
	}_\nu\ensuremath{ \mathbf{Q} }_\sigma + 
	\ensuremath{ \mathbf{P} }_\sigma^{\top}\ensuremath{ \mathbf{P} }_\sigma)^{-1}\ensuremath{ \mathbf{Q} }_\sigma^{\top}\mathbf{\bar{X}}_\mathrm{s}$ 
	positive definite, as $\ensuremath{ \mathbf{Q} }_\sigma^{\top}\mathbf{\bar{X}}_\mathrm{s}$ has full column rank due to Assumption 
	\ref{ass:discretedomain} and 
	$\ensuremath{ \mathbf{Q} }_\sigma^{\top}\ensuremath{ \mathbf{K} }_\nu\ensuremath{ \mathbf{Q} }_\sigma + 
	\ensuremath{ \mathbf{P} }_\sigma^{\top}\ensuremath{ \mathbf{P} }_\sigma$ is positive definite due to Assumption \ref{ass:agauge}.
\end{proof}
Again, the particular circuit case of a voltage or current source is considered and the same Corollary as far the T-$\Omega$ coupling follows 
from the previous Proposition.
\begin{corollary}[Excitation index of the A* formulation]
	The discrete (gauged) system of equations of the A* formulation with circuit coupling equation \eqref{eq:astargauged} has differentiation index
	\begin{itemize}
		\item 1, if the voltage $\ensuremath{ \mathbf{v} }_\mathrm{s} $ is prescribed.
		\item 2, if the current $\ensuremath{ \mathbf{i} }_\mathrm{s}$ is prescribed.
	\end{itemize}
\end{corollary}
This result was already proven in \cite{Bartel_2011aa} for a different gauge and follows now as a Corollary from Proposition \ref{prop:astar} and 
Theorem \ref{theo:index}.

\subsection{Field-Dependent Materials}
In case of having field-dependent materials, the material matrices $\ensuremath{\mathbf{M}_{\mu}}(\mathbf{h})$ for the 
T-$\Omega$ formulation and $\ensuremath{\mathbf{M}_{\nu}}(\mathbf{b})$ for the A* formulation, 
depend on the discretised field quantities $\mathbf{h} = \ensuremath{\mathbf{C}}\mathbf{t}+\ensuremath{\widetilde{\mathbf{S}}}^{\top}\Psi + 
\mathbf{Y}_\mathrm{s}\ensuremath{ \mathbf{i} }_{\mathrm{s}}$ and $\mathbf{b} = \ensuremath{\mathbf{C}}\ensuremath{ \mathbf{a} }$, respectively. 

Taking the time derivative of an equation leads to systems like \eqref{eq:astargauge} and \eqref{eq:tomegauge} with 
differential material matrices $\ensuremath{ \mathbf{M} }_{\mu,\mathrm{d}}$ and $\ensuremath{ \mathbf{M} }_{\nu,\mathrm{d}}$ (see \cite[Chapter 
3]{Romer_2015ab}, 
\cite[Appendix A.3]{Schops_2011ac}) instead of the regular material matrices $\ensuremath{\mathbf{M}_{\mu}}$ and $\ensuremath{\mathbf{M}_{\nu}}$ 
whenever there is a time derivative. For example in the first equation of the T-$\Omega$ formulation
\begin{align*}
\ensuremath{\mathbf{C}}^{\top}\ensuremath{ \mathbf{M} }_\rho\ensuremath{\mathbf{C}}\mathbf{t} + 
\frac{\mathrm{d}}{\mathrm{d}t}(\ensuremath{\mathbf{M}_{\mu}}(\mathbf{h})\mathbf{h})
\end{align*}
applying the chain rule we have
\begin{align*}
\ensuremath{\mathbf{C}}^{\top}\ensuremath{ \mathbf{M} }_\rho\ensuremath{\mathbf{C}}\mathbf{t} + \ensuremath{ \mathbf{M} 
}_{\mu,\mathrm{d}}(\mathbf{h})\frac{\mathrm{d}}{\mathrm{d}t}\mathbf{h}.
\end{align*}
The differential material matrices are built with the differential permeability 
$\mathbf{\mu}_\mathrm{d}(H)$
(respectively the differential reluctivity $\mathbf{\nu}_\mathrm{d}(B)$), where $\lVert \cdot \rVert$ is the Euclidean norm, $H = 
\lVert\ensuremath{\vec{H}}\rVert$ and $B = \lVert\ensuremath{\vec{B}}\rVert$,
\begin{align*}
\mathbf{\mu}_\mathrm{d}(s) &= \mu(\lVert\mathbf{s}\rVert)\ensuremath{ \mathbf{I} } + \frac{1}{s}\frac{\partial 
	\mu(s)}{\partial 
	s}\mathbf{s}\mathbf{s}^{\top} \text{ and }\\
\mathbf{\nu}_\mathrm{d}(s) &= \nu(s)\ensuremath{ \mathbf{I} } + \frac{1}{s}\frac{\partial 
	\nu(s)}{\partial 
	s}\mathbf{s}\mathbf{s}^{\top}, 
\end{align*}
with $\mathbf{s}\in \mathbbm{R}^3$, $s = \lVert\mathbf{s}\rVert$ and $\ensuremath{ \mathbf{I} }\in\mathbbm{R}_{3\times 3}$ the identity tensor of 
second rang. Those tensors are 
positive definite (see 
\cite[Chapter 2]{Pechstein_2004aa}, \cite[Chapter 3]{Romer_2015ab}) under the following natural physical assumptions for the B-H curve $B = 
f_\mathrm{BH}(H)$.
\begin{assumption}[B-H curve \cite{Pechstein_2006aa}]\label{ass:material}
	The B-H curve $$f_\mathrm{BH}(H) = \mu(H)H:\mathbbm{R}_0^+\rightarrow\mathbbm{R}_0^+$$ fulfils
	\begin{itemize}
		\item $f_\mathrm{BH}(s)$ is continuously differentiable.
		\item $f_\mathrm{BH}(0) = 0$.
		\item $f'_\mathrm{BH}(s) \geq \mu_0$, $\forall s>0$.
		\item $\lim_{s\rightarrow\infty}f'(s) = \mu_0$.
	\end{itemize}
	with $\mu_0>0$ being the vacuum permeability.
\end{assumption} 
Analogously to the regular material matrices, we have that the differential material matrices are positive definite provided that the differential 
material tensor is positive definite, which is the case. 
Therefore, the index analysis for the A* and T-$\Omega$ formulations can be analogously transferred to field-dependent materials under 
Assumption \ref{ass:material}.

\begin{prop}[Linear index-2 components]\label{prop:ATgiwli2}
	Both the gauged T-$\Omega$ system \eqref{eq:tomegauge} as well as the A* system \eqref{eq:astargauge} with field-dependent materials have the 
	structure described in Proposition \ref{prop:giwli2} and thus lead to a DAE with linear index-2 components when coupled to a circuit.
\end{prop}

The A* and T-$\Omega$ formulations are complementary, that is, the potentials defined on them live on spaces dual to each other. Also the 
excitations 
imposed on 
the formulations (either current or voltage excited) live on dual spaces (see Figure~\ref{fig:maxhouse}). Therefore, it could be thought that, 
whereas for one formulation it is 
more 
convenient to impose a current excitation, for the other a voltage excitation is better. However, the analysis shows that both behave equally from 
the 
index point of view, as they are inductance-like elements. In both cases a voltage excitation leads to an index-1 system as for the 
case of an inductor and the index results are more connected to the physics described by the DAE rather than to the formulation.
\section{Numerical Results}\label{sec:numeric}

Numerical examples for a field/circuit coupled system using the A* formulation and comparing an index-1 coupled case with an index-2 one have already 
been demonstrated previously (see e.g. \cite{Bartel_2011aa}). 

For the numerical simulations in this paper, 
the T-$\Omega$ system of equations is solved for a model consisting of a square coil and an aluminium core (see Figure~\ref{fig:coil}) and coupled to 
a 
circuit.
\begin{figure}[h]
	\subfigure[Square coil (transparent grey) with iron core (blue).]{	\includegraphics[width=0.45\textwidth]{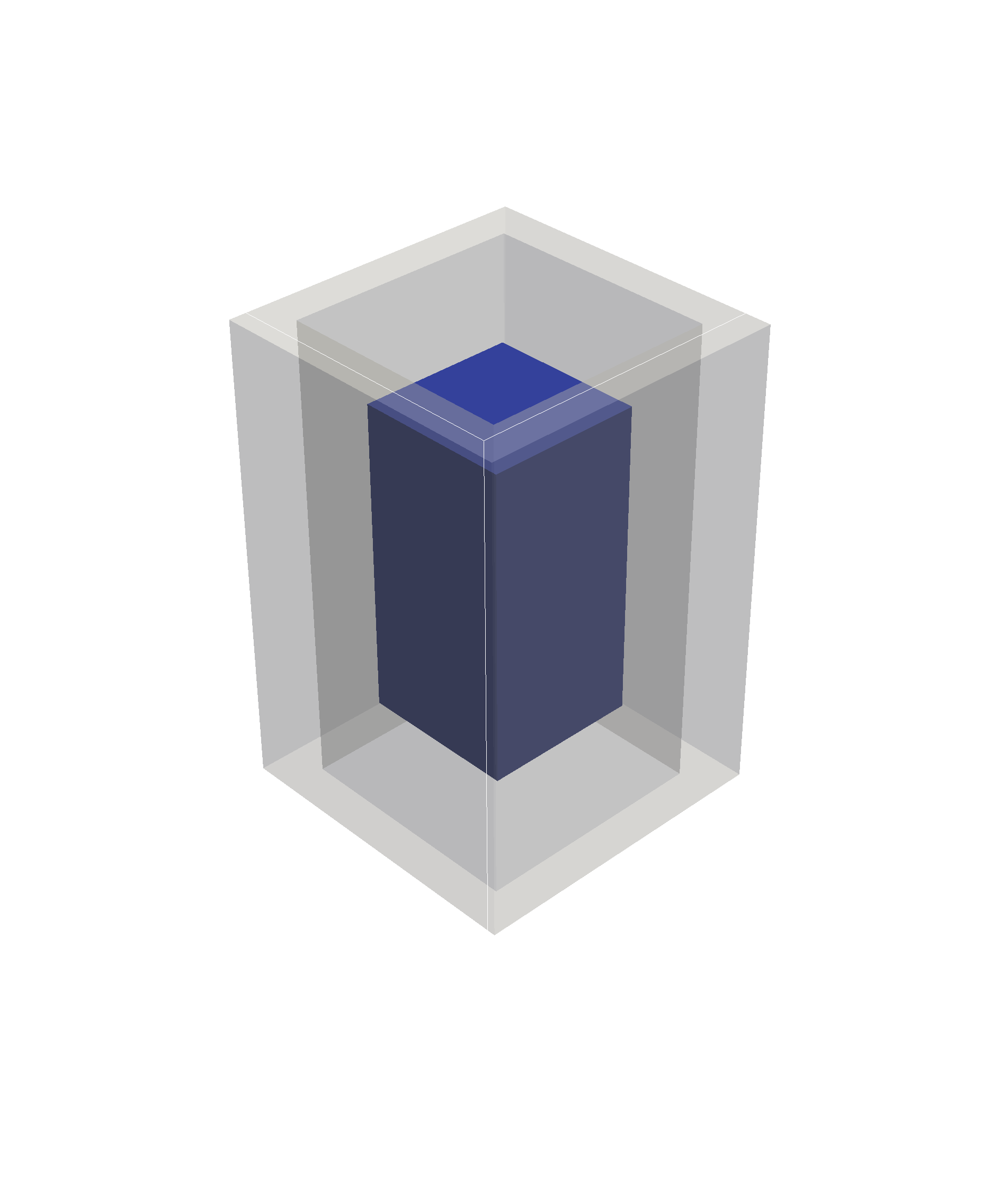}
	}	\hspace{0.06\textwidth}
	\subfigure[Discretised current source $\mathbf{j}_\mathrm{s} = \mathbf{C}\mathbf{Y}_\mathrm{s}$.]{		\includegraphics[width=0.45\textwidth]{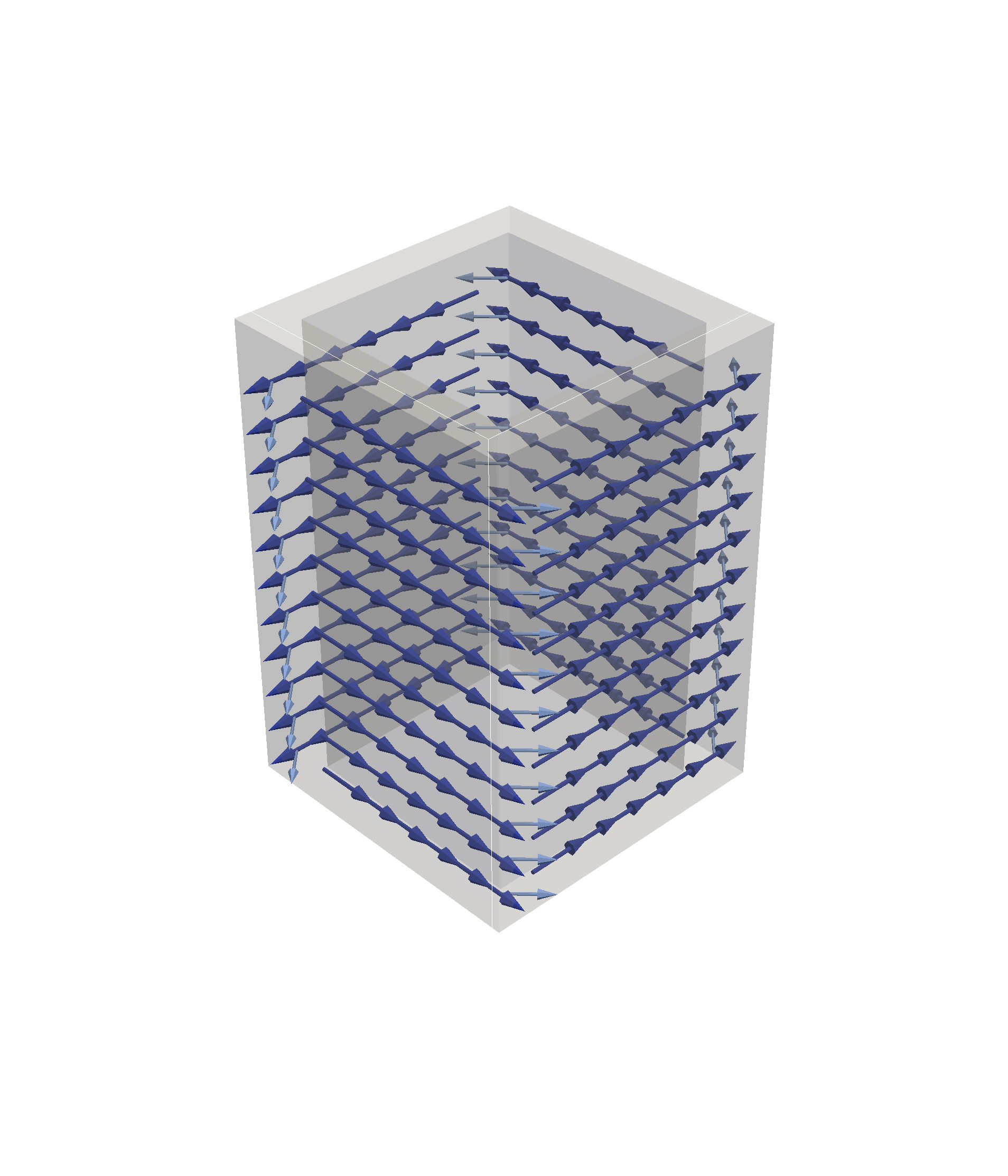}
		}
	\caption{T-$\Omega$ inductance-like element.}\label{fig:coil}
\end{figure}
The discretisation is carried out with the Finite Integration Technique and a tree-cotree gauge is applied to the discretised electric vector 
potential 
in the conducting region (see Section~\ref{sec:T-Omega}).

Two different coupling scenarios are considered. An index-1 case with the magnetoquasistatic device coupled to a voltage source $v_\mathrm{s}(t) = 
\sin(2\pi f_\mathrm{s} t)$ (see Figure~\ref{fig:index1c}) is compared to an index-2 setting (Figure~\ref{fig:index2c}) where the device is coupled to 
a current source
$i_\mathrm{s}(t) = 
\sin(2\pi f_\mathrm{s} t)$, with $f_\mathrm{s} = 2\pi$. In both cases, the simulation is performed first 
with the given excitation $v_\mathrm{s}(t)$ (respectively $i_\mathrm{s}(t)$) and afterwards with a slightly perturbed excitation 
$\tilde{v}_\mathrm{s}(t)$ (respectively $\tilde{i}_\mathrm{s}(t)$), i.e.  
\begin{align*}
\tilde{v}_\mathrm{s}(t) = v_\mathrm{s}(t) + p(t), && \tilde{i}_\mathrm{s}(t) = i_\mathrm{s}(t) + p(t),
\end{align*}
with perturbation
\begin{equation*}
p(t) = \varepsilon_\mathrm{p}\sin(2\pi f_\mathrm{s} \ 10^9t)
\end{equation*}
and $\varepsilon_\mathrm{p} = 10^{-4}$.
\begin{figure}[h]
	\centering
	\subfigure[Index 1: Inductance-like element coupled to voltage source $v_\mathrm{s} = \sin(2\pi f_\mathrm{s} t)$.]{		\includegraphics[width = 0.35\textwidth]{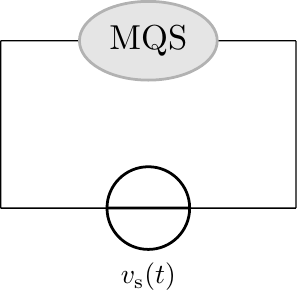}\label{fig:index1c}
	}%
	\hspace{0.1\textwidth}
		\subfigure[Index 2: Inductance-like element coupled to current source $i_\mathrm{s} = \sin(2\pi f_\mathrm{s} t)$.]{		\includegraphics[width = 0.35\textwidth]{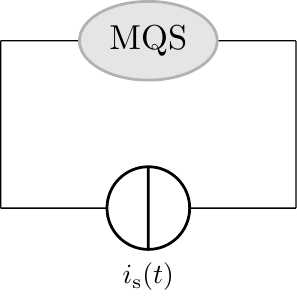}\label{fig:index2c}
	}

	\subfigure[Index 1: Inductance-like element coupled to voltage source $v_\mathrm{s} = \sin(2\pi f_\mathrm{s} t)$ with perturbation 
	$v_\mathrm{p} = 
	\varepsilon_\mathrm{p}\sin(2\pi f_\mathrm{p} t)$.]{		\includegraphics[width = 0.35\textwidth]{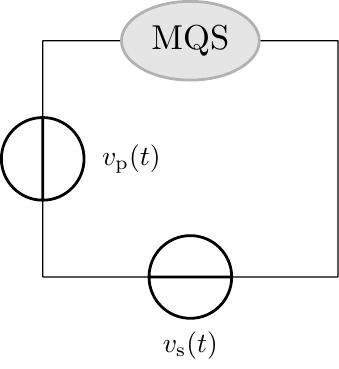}\label{fig:index1p}
	}%
\hspace{0.1\textwidth}
		\subfigure[Index 2: Inductance-like element coupled to current source $i_\mathrm{s} = \sin(2\pi f_\mathrm{s} t)$ with perturbation 
		$i_\mathrm{p} = 
		\varepsilon_\mathrm{p}\sin(2\pi f_\mathrm{p} t)$.]{		\includegraphics[width = 0.35\textwidth]{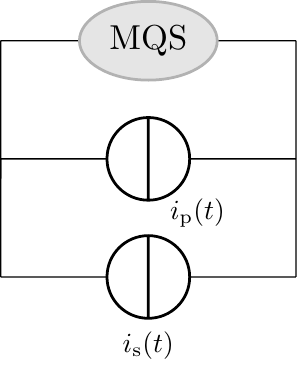}\label{fig:index2p}
	}
	\caption{Different field/circuit coupling schemes.}
\end{figure}
For both cases, an implicit Euler scheme is performed at time interval $\mathcal{I} = [0 \ \  0.5]$ with varying step sizes 
$\delta t = \{8 \cdot 10^{-5},\ 4 \cdot 10^{-5},\, 2 \cdot 10^{-5},\, 10^{-5}\}$. 

Consistent initial conditions on the degrees of freedom 
$\mathbf{x}(t_0) = \mathbf{0}$ are set for the index-1 simulation. In the index-2 case, it has been shown that for DAEs with linear index-2 
components, which is the case for our system (see Proposition \ref{prop:ATgiwli2}),
a consistent initial value is obtained after two implicit Euler iterations \cite{Baumanns_2010aa,Estevez-Schwarz_2000ab}. Here, the starting 
point 
$\mathbf{x}(t_{0}-8\cdot 10^{-5}) = \mathbf{0}$ is selected.

\begin{figure}
	\centering
	\subfigure[Index-1 simulation.]{		\includegraphics[width = 0.49\textwidth]{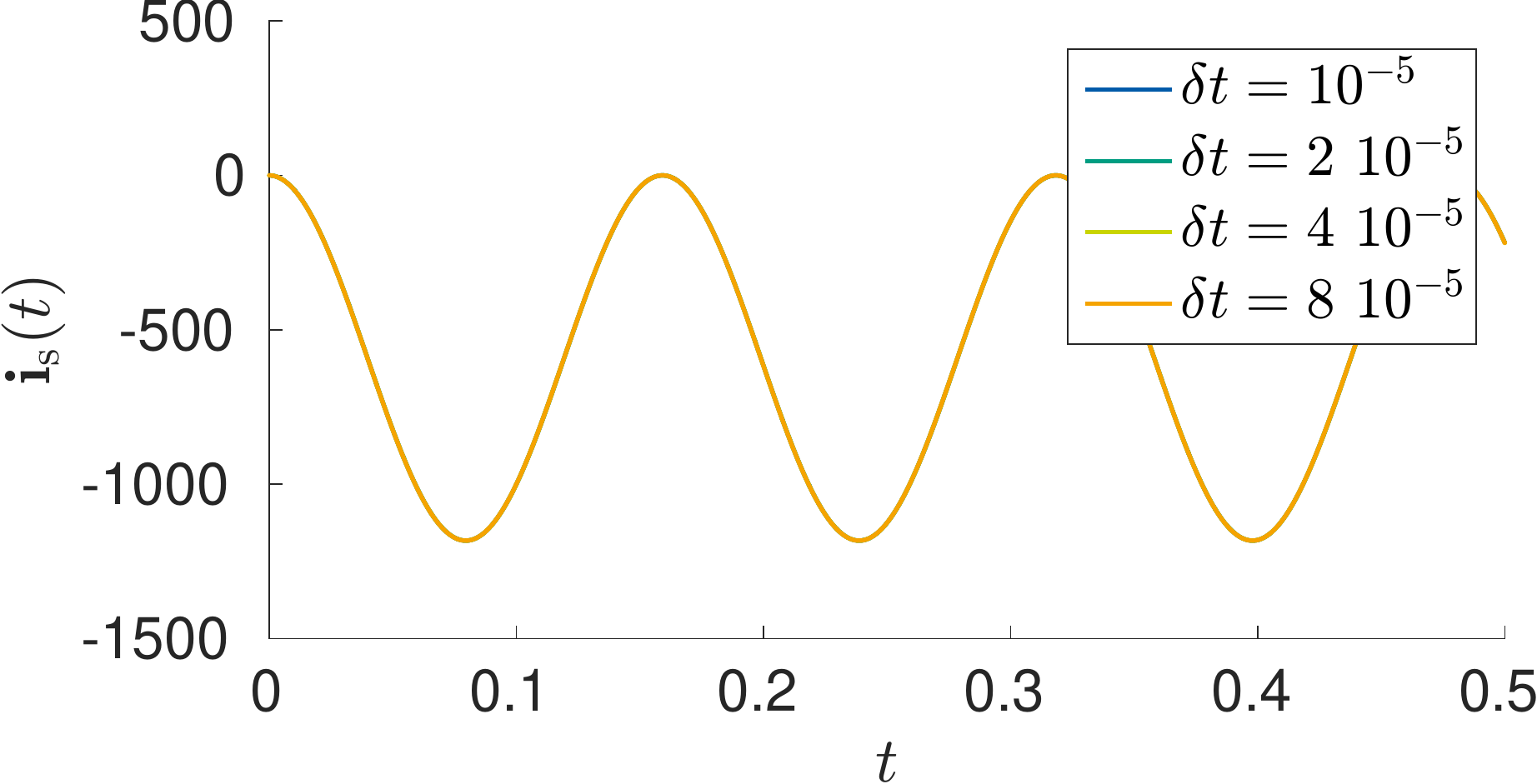}\label{graph:index1c}
	}%
	\subfigure[Index-1 perturbed simulation.]{	\includegraphics[width = 0.49\textwidth]{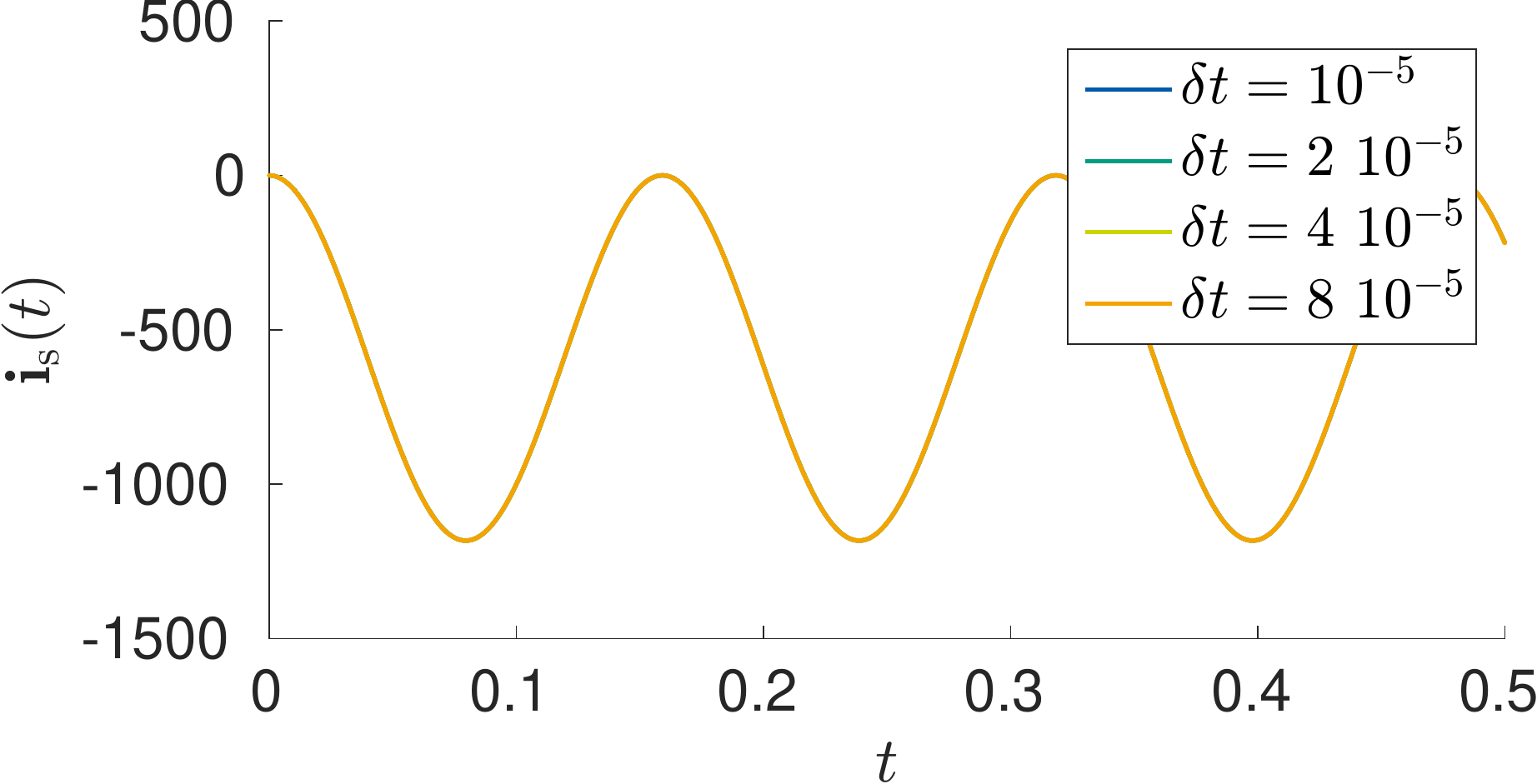}\label{graph:index1p}
}
	\subfigure[Index-2 simulation.]{	\includegraphics[width = 0.49\textwidth]{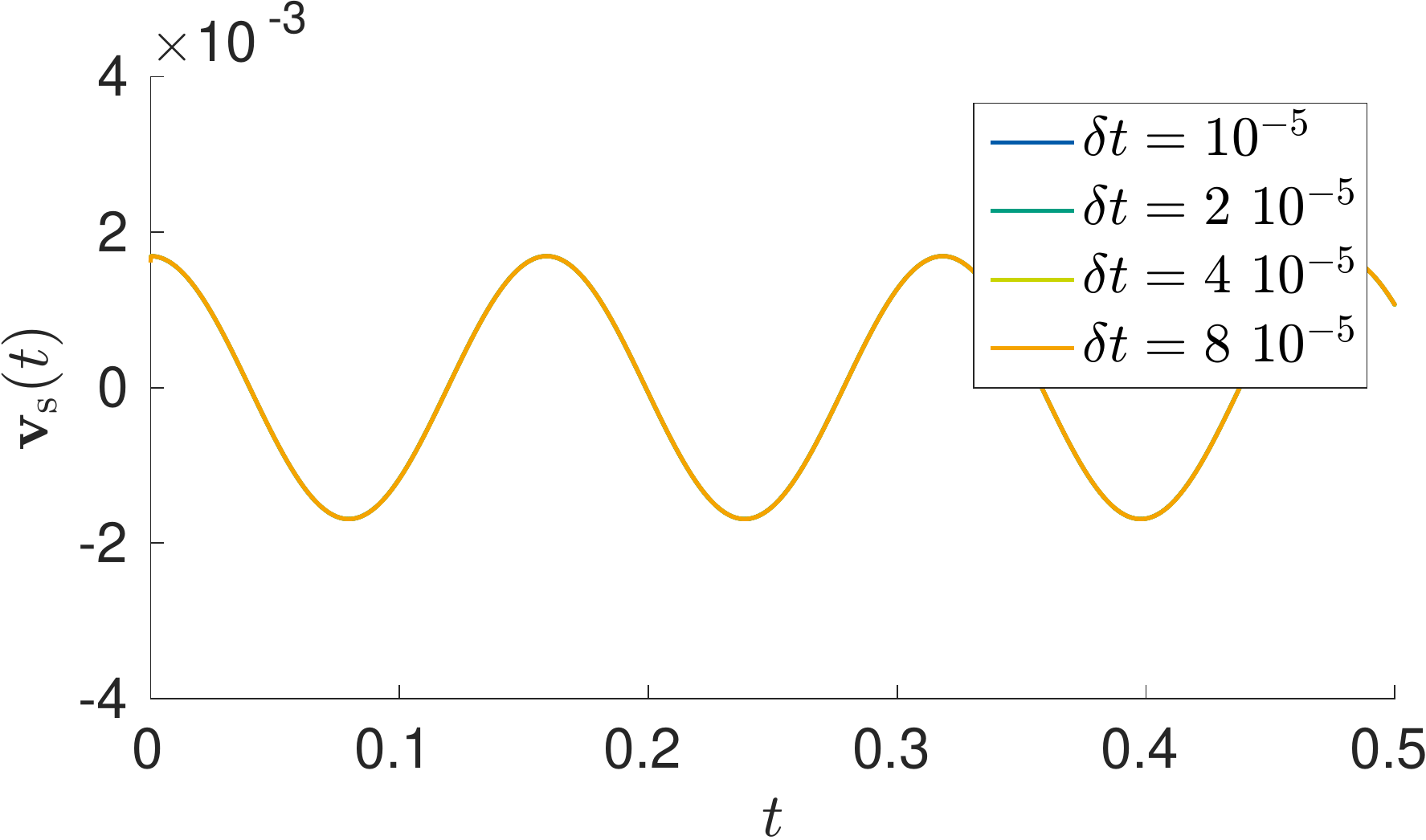}\label{graph:index2c}
}%
\subfigure[Index-2 perturbed simulation.]{	\includegraphics[width = 0.49\textwidth]{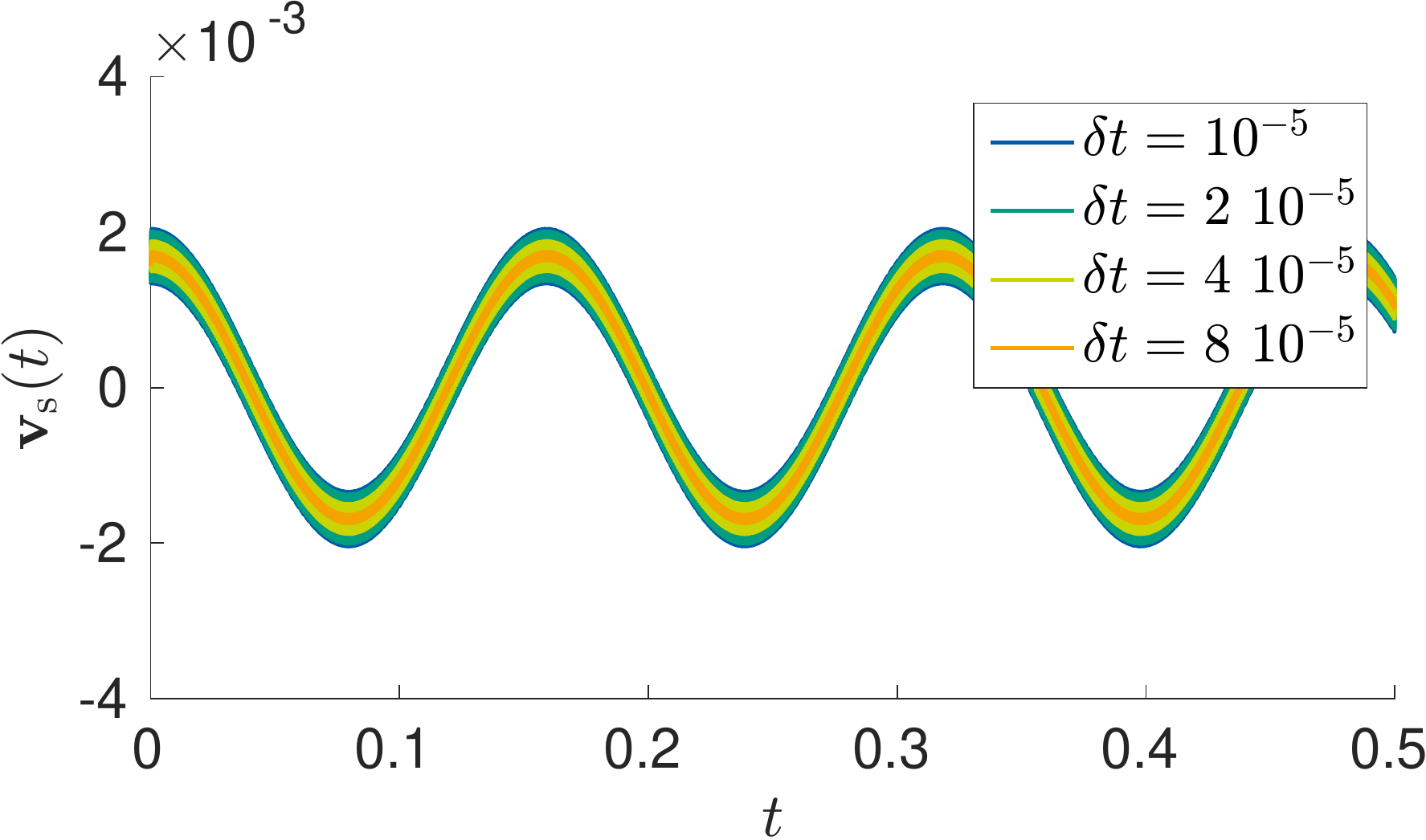}\label{graph:index2p}
}
	\caption{Field/circuit coupled simulation.}\label{fig:graph}
\end{figure}
Figure~\ref{fig:graph} shows the simulation results for both index-1 and index-2  with non-perturbed and perturbed excitations. It can be seen that 
the perturbed index-2 case (Figure~\ref{graph:index2p}) oscillates due to the higher sensitivity of index-2 DAE systems to small 
perturbations, whereas the 
index-1 simulation (Figure~\ref{graph:index1p}) is not significantly affected by the small perturbation on the excitation. 
\section{Conclusions}\label{sec:conclusions}
The paper discusses index results for field/circuit coupled systems for different formulations. The case of the A* formulation was already studied 
previously (see 
\cite{Bartel_2011aa}). However, the index for the field/circuit coupled system of DAEs obtained with a 
T-$\Omega$ formulation had not been analysed before. In order to study the index of the coupled systems, a new generalised element type is 
introduced that from the index point 
of view behaves like an inductor in the MNA formulation. This eases the later index analysis, as now local properties of the DAE system 
describing only the element have to be verified together with topological characteristics of the circuit in order to obtain the index of the entire 
coupled 
system.

Both the A* as well as the T-$\Omega$ formulation field/circuit coupling indexes have been shown to behave like inductances. This yields to the 
conclusion that even though the degrees of freedom of the two formulations and the voltage (respectively the current) excitations live on dual 
spaces, 
prescribing the voltage results in a lower index system in both formulations.

A further study could define a generalised capacitance-like element and derive under which approximations Maxwell's equations embedded as a 
generalised element to a circuit correspond 
to a capacitance-like element.

\begin{acknowledgements}
This work has been supported by the Excellence Initiative of the German Federal and State Governments and the Graduate School of CE at TU Darmstadt.\\
We thank Prof. Caren Tischendorf for the fruitful discussions.
\end{acknowledgements}

\bibliographystyle{spmpsci}      
\bibliography{abbrv,english,bibliography}   

\end{document}